\documentclass[11pt]{amsart}

\usepackage{amsmath, amsthm, amssymb,amscd}
\usepackage{fullpage, paralist,enumerate}
\usepackage{verbatim}
\usepackage[all]{xy}
\usepackage[parfill]{parskip}
\usepackage{graphicx}

\usepackage{afterpage}

\usepackage{amsmath}
\usepackage{amsfonts}
\usepackage{mathcomp}
\usepackage{textcomp}
\usepackage{amssymb}
\usepackage{latexsym}
\usepackage{graphicx}
\usepackage{courier}
\usepackage[english]{babel}
\usepackage{mathbbol}
\usepackage{multirow}
\usepackage{latexsym}
\usepackage{epsfig}
\usepackage{subfigure}
\usepackage{dcolumn}
\usepackage{bm}
\usepackage{colordvi}
\usepackage{color}
\usepackage{booktabs}
\usepackage{stmaryrd}
\usepackage{cite}
\usepackage[linesnumbered,commentsnumbered,ruled]{algorithm2e}
\usepackage{epstopdf}
\usepackage{pdfsync}

\usepackage{hyperref}
\input xy
\xyoption{all}
\newcommand{\commentout}[1]{}

\newtheorem{thm}{Theorem}[section]
\newtheorem{prop}[thm]{Proposition}
\newtheorem{cor}[thm]{Corollary}

\newtheorem{rmk}[thm]{Remark}

\newcommand{\nwc}{\newcommand*}

\nwc{\ben}{\begin{equation*}}
\nwc{\bea}{\begin{eqnarray}}
\nwc{\beq}{\begin{eqnarray}}
\nwc{\bean}{\begin{eqnarray*}}
\nwc{\beqn}{\begin{eqnarray*}}
\nwc{\beqast}{\begin{eqnarray*}}

\nwc{\eal}{\end{align}}
\nwc{\een}{\end{equation*}}
\nwc{\eea}{\end{eqnarray}}
\nwc{\eeq}{\end{eqnarray}}
\nwc{\eean}{\end{eqnarray*}}
\nwc{\eeqn}{\end{eqnarray*}}

\CompileMatrices

\theoremstyle{remark}

\nwc{\nn}{\nonumber}
\nwc{\mb}{\mathbf}
\nwc{\ml}{\mathcal}

\newcommand{\lt}{\left}
\newcommand{\rt}{\right}

\nwc{\vep}{\varepsilon}
\nwc{\ep}{\epsilon}
\nwc{\vrho}{\varrho}
\nwc{\orho}{\bar\varrho}
\nwc{\vpsi}{\varpsi}
\nwc{\lamb}{\lambda}
\nwc{\om}{\omega}
\nwc{\Om}{\Omega}
\nwc{\al}{\alpha}

\nwc{\IA}{\mathbb{A}} 
\nwc{\bi}{\mathbf i}
\nwc{\bo}{\mathbf o}
\nwc{\IB}{\mathbb{B}}
\nwc{\IC}{\mathbb{C}} 
\nwc{\ID}{\mathbb{D}} 
\nwc{\IM}{\mathbb{M}} 
\nwc{\IP}{\mathbb{P}} 
\nwc{\II}{\mathbb{I}} 
\nwc{\IE}{\mathbb{E}} 
\nwc{\IF}{\mathbb{F}} 
\nwc{\IG}{\mathbb{G}} 
\nwc{\IN}{\mathbb{N}} 
\nwc{\IQ}{\mathbb{Q}} 
\nwc{\IR}{\mathbb{R}} 
\nwc{\IT}{\mathbb{T}} 
\nwc{\IZ}{\mathbb{Z}} 

\nwc{\cE}{{\ml E}}
\nwc{\cP}{{\ml P}}
\nwc{\cQ}{{\ml Q}}
\nwc{\cL}{{\ml L}}
\nwc{\cX}{{\ml X}}
\nwc{\cW}{{\ml W}}
\nwc{\cZ}{{\ml Z}}
\nwc{\cR}{{\ml R}}
\nwc{\cV}{{\ml V}}
\nwc{\cT}{{\ml T}}
\nwc{\crV}{{\ml L}_{(\delta,\rho)}}
\nwc{\cC}{{\ml C}}
\nwc{\cO}{{\ml O}}
\nwc{\cA}{{\ml A}}
\nwc{\cK}{{\ml K}}
\nwc{\cB}{{\ml B}}
\nwc{\cD}{{\ml D}}
\nwc{\cF}{{\ml F}}
\nwc{\cS}{{\ml S}}
\nwc{\cM}{{\ml M}}
\nwc{\cG}{{\ml G}}
\nwc{\cH}{{\ml H}}
\nwc{\bk}{{\mb k}}
\nwc{\bn}{{\mb n}}
\nwc{\bp}{{\mb p}}
\nwc{\bz}{\mb z}
\nwc{\by}{\mathbf{h}}
\nwc{\bZ}{\mathbf{Z}}
\nwc{\bF}{\mathbf{F}}
\nwc{\bE}{\mathbf{E}}
\nwc{\bV}{\mathbf{V}}
\nwc{\bY}{\mathbf Y}
\nwc{\br}{\mb r}
\nwc{\pft}{\cF^{-1}_2}
\nwc{\bU}{{\mb U}}
\nwc{\bG}{{\mb G}}
\nwc{\bg}{\mathbf{g}}
\nwc{\mbf}{\mathbf{f}}
\nwc{\mbe}{\mathbf{e}}
\nwc{\be}{\mathbf{e}}
\nwc{\ind}{\operatorname{I}}
\nwc{\mbx}{\mathbf{f}}
\nwc{\bb}{\mathbf{g}}
\nwc{\xmax}{f_{\rm max}}
\nwc{\xmin}{f_{\rm min}}
\nwc{\suppx}{\hbox{\rm supp} (\mbf)}
\nwc{\cI}{\IZ^2_N}
\nwc{\chis}{{\chi^{\rm s}}}
\nwc{\chii}{{\chi^{\rm i}}}
\nwc{\pdfi}{{f^{\rm i}}}
\nwc{\pdfs}{{f^{\rm s}}}
\nwc{\pdfii}{{f_1^{\rm i}}}
\nwc{\pdfsi}{{f_1^{\rm s}}}
\nwc{\thetatil}{{\tilde\theta}}
\nwc{\red}{\color{red}}
\nwc{\blue}{\color{blue}}

\nwc{\prox}{\hbox{prox}}
\nwc{\diag}{\hbox{\rm diag}}
\nwc{\supp}{{\hbox{\rm supp}}}

\nwc{\sloc}{J_{\rm f}}
\nwc{\bu}{\xi}
\nwc{\bv}{\eta}
\nwc{\cU}{\mathcal{U}}
\nwc{\cN}{\mathcal{N}}
\nwc{\bN}{\mathbf{N}}
\nwc{\mbm}{\mathbf{m}}
\nwc{\bw}{\mathbf{w}}
\nwc{\bom}{\mathbf{w}}
\nwc{\bt}{\mathbf{t}}
\nwc{\z}{y}
\nwc{\cY}{\mathcal{Y}}
\nwc{\bM}{\mathbf{M}}
\nwc{\half}{{1\over 2}}
\nwc{\Sf}{S_{\rm f}}
\nwc{\Jf}{J_{\rm f}}
\nwc{\nul}{\hbox{\rm null}_\IR}
\nwc{\spanR}{\hbox{\rm span}_\IR}
\nwc{\Arg}{\hbox{\rm Arg~}}
\nwc{\fdr}{S_{\rm f}}
\nwc{\phase}[1]{\exp\lt[i\measured #1\rt]}

\nwc{\im}{{\rm i}}

\nwc{\tphi}{{{\phi}_0}}
\nwc{\cle}{\preccurlyeq}
\nwc{\modpi}{{{\rm mod}\,2\pi}}
\nwc{\lb}{\llbracket}
\nwc{\rb}{\rrbracket}
\nwc{\sgn}{\mbox{\rm sgn}}

\begin{document}
 \title{Coded Aperture Ptychography: Uniqueness and 
Reconstruction}

\author{ Pengwen Chen
\address{ Applied Mathematics, National Chung Hsing University, Taiwan. Email: {\tt pengwen@nchu.edu.tw} } 
\and Albert Fannjiang 
 \address{
Corresponding author.  Department of Mathematics, University of California, Davis, California,  USA. Email:  {\tt fannjiang@math.ucdavis.edu}
}}
\begin{abstract}
Uniqueness of solution is proved  for any
ptychographic scheme with a random masks under
a minimum overlap condition and  local geometric convergence analysis is given 
for the alternating projection (AP) and Douglas-Rachford (DR) algorithms.
DR is shown to possess a unique fixed point in the object domain and
for AP a simple criterion for distinguishing the true solution among possibly many  fixed points 
is given. 

A minimalist scheme is proposed where the adjacent masks overlap 50\% of area and each pixel of the object is illuminated by exactly four times during the whole  measurement process. 
Such a scheme is conveniently 
parametrized by  the number $q$ of shifted masks in each direction.
The lower bound $1-C/q^2$ is proved for the geometric convergence rate of the minimalist scheme,    predicting a poor performance with large $q$ which is confirmed by numerical experiments. 

Extensive numerical experiments are performed to explore what the general features of a well-performing mask are like, what the best-performing  values of  $q$ for a given mask are,  how robust the minimalist scheme is with respect to measurement noise and what the significant factors affecting
the noise stability are.

\end{abstract}

\maketitle 
\section{Introduction}
X-ray ptychography
 is a coherent diffractive imaging method that uses multiple micro-diffraction patterns obtained through the scan of a localized illumination on the specimen. Ptychographic imaging along with advances in detectors and computation techniques have resulted in optical and electron microscopy  with enhanced resolution without the need for lenses  \cite{ptycho10, supres-PIE, ePIE08,probe09,Chap}.


Ptychography was initially proposed by Hoppe for transmission electron diffraction microscopy \cite{Hoppe1}. In his pioneering work Hoppe showed that recording diffraction patterns at two positions removes the remaining ambiguity between the correct solution and its complex conjugate. Hoppe \cite{Hoppe2} has considered the extension to non-periodic objects with phase-shifting plates as well. 

However, 
only after Faulkner and Rodenburg proposed the so called ptychographical iterative engine (PIE) \cite{PIE104,PIE204,PIE05}, the redundant information collected via  overlapping illuminations was effectively harnessed (see also \cite{Fie08,ePIE08, probe09, ePIE09}).  
A key to success of ptychographic reconstruction is that the adjacent illuminated areas overlap substantially, around
60-70\% in each direction \cite{overlap,ePIE09}. 

The first question for any inverse problems, including ptychography,  is uniqueness of solution. 
This has been resolved in \cite{Iwen} for the
ptychographic scheme  where all possible shifts of a damped and windowed Fourier transform are used, i.e.
with the maximum overlap between adjacent illuminated areas (see more discussion in Section \ref{sec:contribution}).
However, maximum overlap requires overly redundant measurements and hence
the uniqueness question remains for practical ptychographic schemes with a significantly less overlap. 

Another mathematical question surrounding ptychography is convergence analysis of reconstruction algorithms. Few results provide concrete conditions for verifying convergence to the true solution 
and give an explicit estimate for the convergence rate \cite{Iwen,Hesse,ADM}. In particular,
 (global or local) geometric convergence to the {\em true} ptychographic solution has not been established  for any ptychographic reconstruction that assumes less than the maximum overlap between adjacent illuminated areas.

On the other hand, a recurring problem on the technical side in standard ptychography  is an extremely large 
dynamic range with a zero-order component several orders of magnitude more intense than the scattered field. 
To this end, a beam-stop may be introduced to block the zero-order component. Alternatively, 
a randomly phased mask (a diffuser) can be  deployed to  reduce the dynamic range of the recorded diffraction patterns by more than one order of magnitude \cite{rpi,ptycho-rpi,ePIE10}. 

With these motivations, a main purpose of the present work is to
establish the uniqueness theorem for ptychography with a random mask under a minimum 
overlap condition and to prove local geometric convergence to the true solution for the widely used
Alternating Projections (AP) and Douglas-Rachford (DR) algorithms.  Moreover, we give an explicit
bound for the rate of convergence for both algorithms with the minimalist ptychographic scheme
introduced below.


First we describe how each constituent diffraction pattern is measured in our ptychographic scheme.

\subsection{Oversampled diffraction pattern}

 Let $f^0$ be a part of  the unknown object $f$ restricted to the initial subdomain 
 \[
\cM^0=\{\bn = (n_1,n_2)\in \IZ^2: 0\le n_1, n_2\le m\}. 
 \]

Let the Fourier transform of $f$ be written as 
\[
F^0(\bw)=\sum_{\mbm\in \cM^0} e^{-\im 2\pi \mbm\cdot\bw} f^0(\mbm),\quad \bw=(w_1,w_2). 
\]

Under the Fraunhofer 
approximation, the diffraction pattern is proportional to 
$\lt|F^0\lt(\bw\rt)\rt|^2$ which can be written as  \beq
 I^0(\bw)=  \sum_{\bn =-(m,m)}^{(m,m)}\lt\{\sum_{\mbm\in \cM^0} f^0(\mbm+\bn)\overline{f^0(\mbm)}\rt\}
   e^{-\im 2\pi \bn\cdot \bom},\quad \bom\in [0,1]^2.\label{auto}
   \eeq
   \commentout{
   which is the Fourier transform of the autocorrelation
   \beqn
	  R_{ f^0}(\bn)=\sum_{\mbm\in \cM} f^0(\mbm+\bn)\overline{f^0(\mbm)}.
	  \eeqn
	  }
Here and below the over-line notation means
complex conjugacy. 

The expression in the parentheses in \eqref{auto} is the autocorrelation function of $f^0$ and
the summation over $\bn$ takes the form of Fourier transform on  the enlarged  grid
 \begin{equation*}
 \widetilde \cM^0 = \{ (m_1,m_2)\in \IZ^2: -m\le m_1 \le m, -m\le m_2\leq m \} 
 \end{equation*}
which suggests sampling $I^0(\bw)$ 
 on the grid 
\beq\label{L}
\cL = \Big\{(w_1,w_2)\ | \ w_j = 0,\frac{1}{2 m+ 1},\frac{2}{2m+ 1},\cdots,\frac{2m}{2m+ 1}\Big\}. 
\eeq
Let $\Phi^0: \IC^{|\cM^0|}\to \IC^{|\widetilde \cM^0|}$ be the {$\cL$-sampled} discrete Fourier transform (ODFT) defined on $\cM^0$. We can write $I^0(\bw)=|\Phi^0 f^0(\bw)|^2$ for all $\bw\in \cL$.


A randomly coded diffraction pattern measured with a mask is 
the diffraction pattern for 
the  {\em masked object} $
g^0(\bn) =f^0(\bn) \mu^0(\bn)$ 
where the mask function $\mu^0$ is a finite array of random variables.   
With $
\mu^0(\bn)=|\mu^0(\bn)|e^{\im \phi(\bn)}
$
we will focus on the effect of {\em random phase} $\phi$.
For the uniqueness theorem we will assume $\phi(\bn)$ to be  independent, continuous real-valued random variables. In other words, each $\phi(\bn)$ is independently distributed with a probability density function on $[0,2\pi]$ that  may depend on $\bn$.

The continuity assumption on $\phi$ is a technical one for proving {\em almost sure} uniqueness. 
If $\phi$ are discrete random variables, then we would have to settle for uniqueness with high probability. 
  Continuous phase modulation can be experimentally realized with various
techniques depending on the wavelength. See 
\cite{meta1, meta2, meta3, Waller, ptycho-rpi, Spread} for recent innovation and development of random phase modulation techniques.

We also assume that  $|\mu^0(\bn)|\neq 0,\forall \bn\in \cM$ (i.e. the mask is transparent). This is necessary for unique reconstruction of the object
as any opaque pixel of the mask where $\mu^0(\bn)=0$ would block the transmission of the information $f^0(\bn)$. 
By absorbing $|\mu^0(\bn)|$ into the object function we can assume,
without loss of generality, that $|\mu^0(\bn)|=1,\forall\bn\in \cM^0$, i.e. $\mu^0$ represents
a {\em phase} mask.


  \begin{figure}[t]
\begin{center}
\includegraphics[width=10cm]{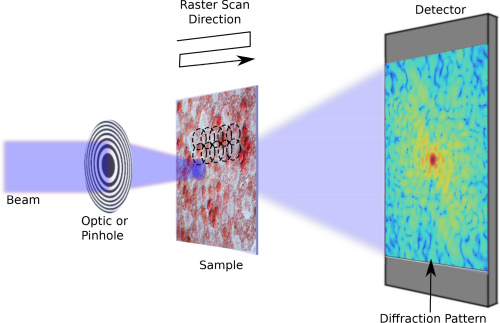}
\caption{Simplified ptychographic setup showing a Cartesian grid used for the overlapping raster scan positions \cite{parallel}.
}
\label{fig0}
\end{center}
\end{figure}

Now consider the simplest, 2-part  ptychographic set-up:  
the object domain is the union of 
two overlapping square grids, one of which is the translate of the other square grid. Denote the two square grids by $\cM^0$ and $\cM^\bt$ 
which is the shift of $\cM^0$ by the displacement vector $\bt=(t_1,t_2)\in \IZ^2$. 
We shall make the overlap assumption
\beq
\label{2.3}
\lt|\cM^0\cap \cM^{\bt}\cap \supp{(f)}\rt|\geq 2,
\eeq
where $\lt|\cdot\rt|$ denotes the cardinality of a set, 
i.e. the intersection of the two grids contains at least two points from the support of the object.

Let 
$f^{\bt}$ be the unknown object restricted to $\cM^\bt$ and $\Phi^\bt$ the  ODFT on $\cM^{\bt}$. We write the object function as $f=f^{0}\cup f^\bt$ where 
$f^{0}(\mbm)=f^\bt(\mbm)$ for all $\mbm\in \cM^0\cap \cM^\bt$.
Let $f^0$ and $f^\bt$ be respectively illuminated with the mask $\mu^0$  and
the  mask $\mu^\bt$ on where  
 $\mu^\bt(\bn)=\mu^0(\bn-\bt), $ for all $\bn\in \cM^\bt$.


For multi-part ptychography, let the object domain be contained in the union of the shifted square grids:\beq\label{1.5}
\supp(f)\subseteq \bigcup_{\bt \in \cT}\cM^{\bt}
\eeq
 where $\cT$ is a set of shifts. 
 We can write
 \beq
 \label{2.6}
 f=\bigcup_{\bt\in \cT} f^\bt.
 \eeq
  Analogous to \eqref{2.3}
we assume that for every $\cM^{\bt_1}\cap\supp(f)\neq \emptyset, \bt_1\in \cT$, there is another 
$\cM^{\bt_2}, \bt_2\in \cT, \bt_2\neq \bt_1$ such that 
\beq\label{2.7}
\lt|\cM^{\bt_1}\cap \cM^{\bt_2}\cap \supp(f)\rt|\geq 2. 
\eeq 
In other words, {\em  every connected component of the object is contained in the union of at least
two distinct masks, whose intersection contains at least two points of the object support. }
Since  the support of the object is often not known precisely,  some illuminations may
totally miss the object and produce no useful information. These illuminations and the resulting data are
easily recognized and should be discarded.

\subsection{A  minimalist ptychographic scheme}
\label{sec:scheme}

Although our uniqueness theorem and local convergence analysis are for
general ptychographic measurement schemes satisfying \eqref{2.7}, we will
consider the following more structured scheme
for numerical experiments and explicit estimation of
convergence rate. We call this practical scheme {\em the minimalist scheme} because each object pixel participates
exactly in four diffraction patterns (two in each direction) and two adjacent mask domains 
have the minimum  (i.e. $50\%$) overlap. 

Suppose the initial mask domain $\cM^0$ is $m\times m$ ($m$ is an even integer) and an adjacent domain is obtained by shifting $m/2$ in either direction.  To cover each object pixel exactly four times (two in each direction) we assume $m= 2n/q$ with an integer $q$. This amounts to $q^2$ diffraction patterns.
In other words, we 
consider  the shift $T^{kl}$ corresponding to the displacement $\bt_{kl}={m\over 2}(k,l)$, with $k\in \{0,1,\cdots, q-1\},~l\in \{0,1,\cdots, q-1\}$. When $k=q-1$ or $l=q-1$ we
assume for simplicity that the shifted mask is wrapped around into the other end of the object domain (i.e. the periodic boundary condition). Four times coverage  and four times oversampling
in each diffraction pattern  together  produce the total number $q^2(2m-1)^2\approx 16n^2$ of data. 
 
 We emphasize two features of the minimalist scheme: (i) The scheme has a fixed total oversampling ratio ($\approx 16$) independent of the total number of shifted masks, $q^2$; (ii) This scheme has the minimum ($50\%$) 
 overlap between two adjacent masks while maintaining the same number of coverage
 (i.e. 4) for every pixel of the object. Note that the $50\%$ overlap is lower than the required overlap found empirically  (i.e. $60\%-70\%$) from
 previous studies  \cite{overlap,ePIE09}. 
 
 These features are highly desirable in practice because an efficient  ptychographic scheme should not only collect
 as few data  as possible overall but also  deploy as few shifted masks as possible.
  
With this minimalist ptychographic scheme, we study
 analytically and numerically how $q$ and the nature of the mask affect ptychographic  reconstruction.

 \subsection{Main contributions}
 \label{sec:contribution}
 
 The first result of the present work is the almost sure uniqueness of ptychographic solution with an independent random mask under
 the minimum overlap condition \eqref{2.7}  (Theorem \ref{thm:u} and Corollary \ref{cor:u}).

In this connection,  Iwen {\em et al.}  \cite{Iwen} proved a uniqueness theorem for the
ptychographic scheme  where all possible shifts of a damped and windowed Fourier transform are used, i.e.
the overlap percentage between two adjacent mask domains is at the maximum. In our notation, this amounts to
$n^2$ oversampled diffraction patterns totaling $(2m-1)^2n^2$ number of data. 
Their uniqueness theorem also holds with probability $1-\cO(\ln^{-2}{n}\ln^{-3}(\ln n))$ after randomly selecting
a subset of $\cO(n^2\ln^{2}{n}\ln^{3}(\ln n))$ data (assuming $m$ is at least poly-log in $n$).

In Section \ref{sec:ITA}, we establish local, geometric convergence (Theorem \ref{thm1}) for the alternating
projections (AP) and the Douglas-Rachford (DR) algorithms under the minimum overlap condition \eqref{2.7}. 
We also prove the uniqueness of the DR fixed point in the object domain (Proposition \ref{prop4})
and give an easily verifiable criterion for distinguishing the true solution among many AP fixed points  (Proposition \ref{prop:ap}).

In comparison, Wen {\em et al.} \cite{ADM} proposed alternating direction methods (ADM), including DR,  for ptychographic
reconstruction and demonstrated good numerical performance.
Hesse {\em et al.} \cite{Hesse} 
proved 
global convergence to a critical point for a proximal-regularized alternating minimization formulation of blind ptychography. Many critical points, however,  may co-exist and there is no easy way of distinguishing the true solution from the rest.
Neither of these papers establishes  uniqueness, convergence to the {\em true} solution or the geometric  sense of convergence. 

We also give a bound on the convergence rate of AP and DR for the minimalist scheme introduced 
in Section \ref{sec:scheme} (Proposition \ref{prop:gap}). The bound shows that the convergence rate
can deteriorate rapidly as $q$ becomes large, indicating that the best performing $q$ are in the small and
medium ranges. Our numerical experiments in Section \ref{sec:num} bear this prediction out nicely, focusing
on two kinds of masks: (independent or correlated) random masks and the Fresnel mask.

Finally we prove that for $q=2$, twin image exists in  ptychography with the Fresnel mask at certain values of the Fresnel number and causes the reconstruction error to spike (Propositions \ref{fo} and \ref{prop:twin}). 
As the Fresnel mask is most convenient to fabricate, this result gives a guideline for avoiding poor-performing
masks. 

We summarize our numerical findings  in the Conclusion (Section \ref{sec:con}).

\section{Uniqueness of ptychographic solution}\label{sec:u}

First we recall some basic results from {\em nonptychographic}  phase retrieval where the mask $\mu$  
and 
the unknown object $f$ have the same dimension.

The  $z$-transform
$$F(\bz) = \sum_{\bn} f(\bn) \bz^{-\bn}$$
of $f$ is a polynomial in $\bz^{-1}$ and 
can be 
factorized  uniquely into the product of
irreducible polynomials $F_k(\bz)$ and a monomial in $\bz^{-1}$  \beq
\label{temeq2}
F(\bz) = \alpha \bz^{-\bn_0} \prod_{k=1}^p F_k(\bz),
\eeq
where $\bn_0$ is a vector of nonnegative integers and
 $\alpha$ is a complex coefficient. 
 

\begin{prop}

\label{prop1}\cite{Hayes}
Let the $z$-transform $F(\bz)$ of a finite complex-valued  array $\{f(\bn)\} $ 
 be
given by
\beq
F(\bz)=\alpha \bz^{-\mbm} \prod_{k=1}^p F_k(\bz),\quad  \mbm\in \IN^d, \alpha\in \IC\label{21}
\eeq
where $F_k, k=1,...,p$ are 
nontrivial irreducible polynomials. Let $G(\bz)$ be
the $\bz$-transform of another finite array $g(\bn)$. 
Suppose 
$|F(\bw)|=|G(\bw)|,\forall \bw\in [0,1]^d$. Then $G(\bz)$ must have
the form
\beqn
G(\bz)=|\alpha| e^{\im \theta} \bz^{-\bp}
\lt(\prod_{k\in I} F_k(\bz)\rt)
\lt(\prod_{k\in I^c} F_k^*(1/\bz^*)\rt),\quad\bp\in \IN^d,\theta\in \IR
\eeqn
 where $I$ is a subset of $\{1,2,...,p\}$. 
\end{prop}
{ {\bf Line object:}  $f$ is a {\em line object}  if the convex hull of  the object support in $\IR^d$  is a line segment.}

\begin{prop}\cite{unique}
Suppose $f$ is not a line object and let the mask $\mu$'s phase be continuously and independently distributed. 
Then with probability one the only irreducible factor of 
the $z$-transform of the masked object $\tilde f(\bn) =f(\bn) \mu(\bn)$ is a  monomial of $\bz^{-1}$. 
\label{prop2}
\end{prop}

The following uniqueness theorem is our first theoretical result.

\begin{thm} 
\label{thm:u} 
Suppose that the assumptions of Proposition \ref{prop2} hold and that
\[
\lt|\cM^0\cap \cM^\bt\cap \supp{(f)}\rt|\geq 2.
\]  
 Then with probability one $f$ is uniquely determined,  up to a global phase factor, by the measurement data $b=|A^*f|$. 

\end{thm}
\begin{proof}
Let  $g(\bn)$ be another array that vanishes  outside $\cM^0\cup \cM^\bt$  and produces the same masked Fourier magnitude data. By Proposition \ref{prop1} and \ref{prop2}, $g$ has the following possibilities: In $\cM^0$, $g$ has two alternatives 
\beq
\label{1.6}
g(\bn)&=&\lt\{\begin{matrix} e^{\im \theta_1} f^0(\bn + \mbm_1)\mu^0(\bn + \mbm_1)/\mu^0(\bn) \\
e^{\im \theta_1} \bar f^0(\bN-\bn + \mbm_1)\bar \mu^0(\bN-\bn + \mbm_1)/\mu^0(\bn),
\end{matrix}\rt. \quad\forall \bn \in \cM^0
\eeq
and
\beq
\label{1.7}
g(\bn)&=&\lt\{\begin{matrix} e^{\im \theta_2} f^\bt(\bn + \mbm_2)\mu^\bt(\bn + \mbm_2)/\mu^\bt(\bn) \\
e^{\im \theta_2} \bar f^\bt(\bN-\bn + \mbm_2)\bar\mu^\bt(\bN-\bn + \mbm_2)/\mu^\bt(\bn),
\end{matrix}\rt. \quad\forall \bn \in \cM^\bt
\eeq
for  some $\mbm_1,\mbm_2\in \IZ^d, \theta_1,\theta_2\in \IR$. 

We now focus on the intersection $\cM^0\cap \cM^\bt$ where \eqref{1.6} and \eqref{1.7} are
both defined. We have then four scenarios from the crossover of the alternatives in \eqref{1.6} and \eqref{1.7}.  

First of all, if, for all $\bn\in \cM^0\cap \cM^\bt$,
\beq\label{1.8}
g(\bn)&=&e^{\im \theta_1} f^0(\bn + \mbm_1 )\mu^0(\bn + \mbm_1)/\mu^0(\bn)\\
&=&e^{\im \theta_2} f^\bt(\bn + \mbm_2)\mu^\bt(\bn + \mbm_2)/\mu^\bt(\bn) \nn
\eeq
then
\beq
\label{1.9}
e^{\im\theta_1}f^0(\bn + \mbm_1 )\mu^0(\bn + \mbm_1)/\mu^0(\bn)
=e^{\im \theta_2}  f^\bt(\bn + \mbm_2 )\mu^0(\bn-\bt + \mbm_2)/\mu^0(\bn-\bt).
\eeq
Clearly, $f(\bn + \mbm_1 )$ and $f^\bt(\bn + \mbm_2 )$ must 
simultaneously be zero or nonzero. When they are nonzero,   we obtain by taking logarithm on both sides \beq
\label{1.10}
\lefteqn{\im\theta_1+ \ln{f^0(\bn+\mbm_1)}+\ln{\mu^0(\bn+\mbm_1)}+\ln{\mu^0(\bn-\bt)}}\\
&=&\im\theta_2+\ln{f^\bt(\bn+\mbm_2)}+\ln{\mu^0(\bn-\bt+\mbm_2)} +\ln{\mu^0(\bn)}\nn
\eeq
which holds up to a multiple of $2\pi$. 
The four random variables
\[
\ln{\mu^0(\bn+\mbm_1)},~\ln{\mu^0(\bn-\bt)},~\ln{\mu^0(\bn-\bt+\mbm_2)},~\ln{\mu^0(\bn)}
\]
can not cancel one another
 unless either $\mbm_1=\mbm_2=0$ or $(\bt=0~\&~\mbm_1=\mbm_2)$. 
When the continuous random variables do not cancel one another,  \eqref{1.10} fails to hold true almost surely. 

On the other hand, for $\mbm_1=\mbm_2=0$ (since $\bt\neq 0$), it follows from \eqref{1.8} that 
\[
g(\bn)=e^{\im\theta_1}f^0(\bn  )
=e^{\im \theta_2}  f^\bt(\bn),\quad\bn\in \cM^0 \cap \cM^\bt.
\]
Since $f^0(\bn)=f^{\bt}(\bn)$, we have  
$\theta_1=\theta_2$. It follows then from \eqref{1.6}-\eqref{1.7} that 
$g=e^{\im\theta_1}(f^0\cup f^\bt)$. 

The other three scenarios can be similarly dealt with.
Consider the next scenario where for $\bn\in \cM^0\cap \cM^\bt$ 
\beqn
g(\bn)&=&e^{\im \theta_1} f^0(\bn + \mbm_1 )\mu^0(\bn + \mbm_1)/\mu^0(\bn)\\
&=&e^{\im \theta_2} \bar f^\bt(\bN-\bn + \mbm_2)\bar \mu^\bt (\bN-\bn + \mbm_2)/\mu^\bt(\bn).
\eeqn
Taking logarithm and rearranging terms we have
\beq\label{1.12}
\lefteqn{\im\theta_1+ \ln{f^0(\bn+\mbm_1)}+\ln{\mu^0(\bn+\mbm_1)}+\ln{\mu^0(\bn-\bt)}+\ln{\mu^0(\bN-\bn-\bt+\mbm_2)}}\\
&=&\im\theta_2+\ln{\bar f^\bt(\bN-\bn+\mbm_2)} +\ln{\mu^0(\bn)}\nn
\eeq
The four random variables
\[
\ln{\mu^0(\bn+\mbm_1)},~\ln{\mu^0(\bn-\bt)},~\ln{\mu^0(\bN-\bn-\bt+\mbm_2)},~\ln{\mu^0(\bn)}
\]
can not cancel one another
since $\bt\neq 0$. As a result,  \eqref{1.12} holds true with probability zero.

The argument for ruling out the  third scenario  
\beqn
g(\bn)&=& e^{\im \theta_1} \bar f^0(\bN-\bn + \mbm_1)\bar \mu^0(\bN-\bn + \mbm_1)/\mu^0(\bn)\\
&=&e^{\im \theta_2} f^\bt(\bn + \mbm_2)\mu^\bt(\bn + \mbm_2)/\mu^\bt(\bn)  
\eeqn
is the same as for the second scenario.

Now consider the fourth scenario 
\beqn
\nn
g(\bn)&=&e^{\im \theta_1} \bar f^0(\bN-\bn + \mbm_1)\bar \mu^0(\bN-\bn + \mbm_1)/\mu^0(\bn)\\
&=& e^{\im \theta_2} \bar f^\bt (\bN-\bn + \mbm_2)\bar \mu^\bt (\bN-\bn + \mbm_2)/\mu^\bt(\bn)\eeqn
which after taking logarithm and rearranging terms becomes
 \beq
\label{1.13}
\lefteqn{\im\theta_1+ \ln{\bar f^0(\bN-\bn+\mbm_1)}+\ln{\mu^0(\bN-\bn-\bt+\mbm_2)}+\ln{\mu^0(\bn-\bt)}}\\
&=&\im\theta_2+\ln{\bar f^\bt(\bN-\bn+\mbm_2)}+\ln{\mu^0(\bN-\bn+\mbm_1)} +\ln{\mu^0(\bn)}.\nn
\eeq
Since $\bt\neq 0$, the four random variables
\[
\ln{\mu^0(\bN-\bn-\bt+\mbm_2)},~\ln{\mu^0(\bn-\bt)},~\ln{\mu^0(\bN-\bn+\mbm_1)},~\ln{\mu^0(\bn)}
\]
cancel one another  only when 
\beqn
\bN-\bn-\bt+\mbm_2&=&\bn\\
\bn-\bt&=&\bN-\bn+\mbm_1
\eeqn
or equivalently
\beqn
\mbm_2&=&2\bn-\bN+\bt\\
\mbm_1&=&2\bn-\bN-\bt
\eeqn
\commentout{
\beq
{\mbm_2-\mbm_1\over 2}&=& \bt\\
{\mbm_2+\mbm_1\over 2} &=& 2\bn-\bN
\eeq
}
which can not hold true simultaneously for more than one $\bn$ for any given $\mbm_1,\mbm_2$. 
This is ruled out by the assumption that $\cM^0\cap \cM^\bt\cap \supp(f)$ contains at least two points. 

In summary, the only possibility is  that 
\[
g=e^{\im\theta}(f^0\cup f^\bt)=e^{\im \theta} f
\]
for some $\theta\in \IR$, which is what we set out to prove.

\end{proof}

The divide-overlap-and-conquer strategy is readily extendable to the multi-part setting.

\begin{cor}
\label{cor:u}
 Consider the multi-part ptychography \eqref{1.5}-\eqref{2.6}. Suppose that the assumptions of Proposition \ref{prop2} hold and 
 that for every $\cM^{\bt_1}\cap\supp(f)\neq \emptyset, \bt_1\in \cT$ there is another $\bt_2\in \cT, \bt_2\neq \bt_1,$ such that 
\beq\label{two2}
\lt|\cM^{\bt_1}\cap \cM^{\bt_2}\cap \supp(f)\rt|\geq 2. 
\eeq 
Then with probability one $f$ is determined uniquely, up to a constant phase factor for each connected component of $f$, by
the ptychographic data 
\beq
\label{1.19}
\lt\{\lt|\Phi^\bt(\mu^\bt\odot f^{\bt})\rt|:\bt\in \cT\rt\}.
\eeq

The constant phase factor becomes global, i.e. the same one for the whole object, if
\beq
\label{connect}
\bigcup \lt\{ \cM^\bt:\cM^\bt\cap f\neq\emptyset, \bt\in \cT\rt\}\quad\mbox{is a connected set}. 
\eeq
\end{cor}
\begin{rmk}
Clearly the result still holds when some of the shifted masks do not intersect with the object. This has
a practical relevance as  the support of the object is often not precisely known and some illuminations can
totally miss the object. Of course, these illuminations produce no useful information and should be discarded.

When the condition \eqref{connect} fails, the whole ptychographic problem breaks down into a set of
separate independent subproblems, each with the data corresponding to a connected component of 
$\bigcup \lt\{ \cM^\bt:\cM^\bt\cap f\neq\emptyset, \bt\in \cT\rt\}$. 

\end{rmk}

\begin{proof}
Let  $\bt_1,\bt_2\in \cT$ be any pair of shifts satisfying the overlapping property \eqref{two2}. Then by Theorem \ref{thm:u}, $f^{\bt_1}\cup f^{\bt_2}$ is uniquely determined, up to a constant phase factor, by the data 
\[
\lt\{\lt|\Phi^{\bt_j}(\mu^{\bt_j}\odot f^{\bt_j})\rt|: j=1,2\rt\} 
\]
with probability one where $\odot$ denotes the Hadamard (i.e. componentwise) product. 
Since $f$ is the union of all such pairs $f^{\bt_1}\cup f^{\bt_2}$, $f$ is uniquely determined, up to a constant phase factor for each connected component of $f$, by the data \eqref{1.19}, with probability one. 

The constant phase factor ambiguity for every connected component may not be the same  since some masks may have no intersection with the object. Once condition \eqref{connect} is valid, 
the constant phase factor must be the same for the whole object. 
\end{proof}

\section{Fixed point algorithms}\label{sec:ITA}

To describe the reconstruction algorithms, it is most convenient to resort to the vector-matrix notation
where we use $\IC^{N}$ ($N=$ the total number of pixels in the object = $n^2$) as the object space and $\IC^{M}$ ($M=$ the total number of measurement data $=m^2$) as the the data space
{\em before} taking the modulus of the diffracted field. We use $\|\cdot\|$ to denote
the vector norm as well as the Frobenius norm when the object is written as a matrix. 

A phase-masked measurement gives rise to an  {\em isometric} matrix in
the non-ptychographic setting
\beq
\label{one}
\hbox{\rm (1-pattern nonptychographic matrix)}\quad A^*= c\Phi\,\, \diag\{\mu\},
\eeq
where the constant $c$ is chosen to normalize $A^*$ such that $AA^*=I$. 
The 2-pattern ptychography matrix $A^*$  can be written as 
\beq \label{two}\hbox{(2-pattern ptychography matrix)}\quad 
A^*=c \lt[\begin{matrix}
\Phi^0~\diag\{\mu^0\} &0 \\
0&\Phi^\bt ~\diag\{\mu^\bt\}
\end{matrix}\rt]
\eeq
where the first and second mask domains overlap due to the nature of a ptychographic scheme. 

The propagation matrix $A^*$ for multi-part ptychography is constructed analogous to \eqref{two} by
 stacking $\Phi^\bt~\diag\{\mu^\bt\},~\forall \bt \in \cT$ in the proper order. 
For  algorithmic analysis, we normalize  
the columns of  $A^*$ so that $A^*$ is isometric.

Let $b\equiv |A^* f| \in \IR^M$. 
 For any $y\in \IC^M$, $\sgn(y)\in \IC^M$ is defined as 
\[
\sgn(y)[j]=\lt\{\begin{matrix} 
1&\mbox{if $y[j]=0$}\\
y[j]/|y[j]|&\mbox{else. }
\end{matrix}\rt.
\]

Ptychography  can be formulated   as the following feasibility problem 
in the Fourier domain 
 \beq
 \label{feas}
\hbox{Find}\quad  \hat y\in  A^*\cX \cap \cY,\quad \cY:=  \{y\in \IC^M: |y|=b\}. 
 \eeq

Let $P_1$ be the projection onto $A^*\cX$ and $P_2$ the projection onto $\cY$:
\beq
\label{proj}
P_1 y=A^*Ay,\quad P_2 y=b\odot{\sgn(y)}. 
\eeq

The following are two of the most widely used iterative algorithms for solving feasibility problems. 
\begin{itemize}
\item[\bf Alternating projections]

 \begin{equation}\label{papf}
y^{(k+1)}=P_1P_2 y^{(k)}
\end{equation}

\item[\bf Douglas-Rachford  algorithm]

 

\begin{eqnarray}
y^{(k+1)}&=&y^{(k)}+P_1 (2P_2 - I) y^{(k)}- P_2 y^{(k)}
\label{fdr}
\end{eqnarray}
\end{itemize}

As the final output of either algorithm, the object estimate is given by $x^{(k)}=Ay^{(k)}$. 

As we discuss below, AP and DR have their respective strengths and weaknesses and we will combine their strengths in ptychographic reconstruction. 

\subsection{Fixed point}

To accommodate the arbitrariness of the phase of zero components, we call  $y_*$ a {\em Fourier-domain DR fixed point} if there exists 
\[
u\in U=\{
u=(u[i])\in \IC^M: |u[i]|=1,\,\, \forall i\}\
\] satisfying  \beq
\label{fix2}
u\in U,\quad u[j]=1,\quad \hbox{whenever}~y_*[j]\neq 0
\eeq
such that the DR fixed point equation holds
\beq
\label{fixed}
A^*A \lt(2b\odot \sgn({\z_*})\odot u-\z_*\rt)=b\odot \sgn({\z_*}) \odot u.\label{fdru}
\eeq 
Note that  if  the sequence  of iterates $y^{(k)}$ converges a limit $y_\infty$ that has no zero component, then the limit $y_\infty$ is a Fourier domain DR fixed point with $u\equiv 1$. 

Let $ x_*=Ay_* $ be the corresponding object-domain fixed point. 
Define another object estimate 
\beq
\hat x&=&A\left(2b\odot \sgn(y_*)\odot u-y_*\right)
\label{14}\eeq
for some $u$ satisfying \eqref{fix2}.  

We have from \eqref{fixed} 
\beq
A^* \hat x= b\odot \sgn(y_*) \odot u\eeq
which implies \beq
|A^* \hat x|&=& |A^* f|\label{15}\\
\arg( A^* \hat x) &=&\arg(\sgn(y_*)\odot u) \quad\hbox{on}~\supp(b).\label{16}
\eeq

\begin{prop} \label{prop4}
Under the assumptions of Corollary \ref{cor:u} including \eqref{connect},  $\hat x=x_*=e^{\im\theta}f$ for some constant $\theta\in \IR$ almost surely. 

\end{prop}

\begin{proof}

By Theorem \ref{thm:u} and Corollary \ref{cor:u},  (\ref{15}) implies that
$\hat x=e^{\im \theta}f$ for some constant $\theta\in \IR$. To complete the proof, we only need to  show $e^{\im \theta} f=x_*$.\\

By \eqref{16} and the identity $\hat x=e^{\im \theta}f$, we have
\beq
\label{100}
e^{\im \theta}\sgn(A^*f)=\sgn(y_*)\odot u \quad\hbox{on}~\supp(b). 
\eeq
Substituting \eqref{100} into \eqref{14} we obtain
\[
e^{\im \theta} f=A\left(2b\odot e^{\im \theta}\sgn(A^*f)-y_*\right)= 2 e^{\im\theta} A\left(b\odot\sgn(A^*f)\rt)-Ay_*= 2e^{\im \theta} f-x_*
\]
where the last identity follows from the isometry of $A^*$ and the definition of $x_*$.
Hence  $e^{\im \theta} f=x_*$ as claimed. 
\end{proof}

Likewise, we call $x_*$ an AP fixed point  if  for some $u\in U$
 \beq
x_*=\label{3.7.2}
A\lt(b\odot  u \odot \sgn({A^*x_*})\rt). 
 \eeq 
The following result identifies  any AP limit point   
with an AP fixed point. 

\begin{prop}\label{Cauchy} Under the assumptions of Corollary \ref{cor:u}, every limit point  of AP iterates $\{x^{(k)}\} $  is an AP fixed point in the sense \eqref{3.7.2}.
\end{prop}
The proof of Proposition \ref{Cauchy} can be adapted from \cite{AP-phasing} {\em verbatim} and is omitted. 

How do we  distinguish  the true ptychographic solution from the possibly many AP fixed points
(in view of recurring numerical stagnation from random initialization)? 

Consider the inequality \beq
\label{317}
\|x_*\|= \lt\|A\lt(\sgn\{{A^* x_*}\}\odot b\odot u\rt) \rt\| \le \lt\|\sgn\{A^* x_*\}\odot b\odot u\rt\|=\|b\|.
 \eeq
Clearly $\|x_*\|=\|b\|$ holds if and only if  the inequality in Eq.~(\ref{317})  is an equality, which is true only when     \begin{equation}\label{11'}
 \sgn\{A^* x_*\}\odot b\odot u= A^* z \quad \textrm{ for some $z \in \IC^n$}. 
 \end{equation}  
 Since $AA^*=I$ the fixed point equation \eqref{3.7.2} implies  $z=x_*$ and hence  \[
 \sgn\{A^* x_*\}\odot b\odot u=A^* x_*. 
 \]
Thus $b=|A^* x_*|$ implying $x_*$ is the ptychographic solution by Corollary \ref{cor:u}.
Therefore 

\begin{prop}\label{prop:ap} Under the assumptions of Corollary \ref{cor:u} including \eqref{connect}, 
all AP fixed points $x_*$ satisfy $\|x_*\|\le \|b\|$ and $x_*=f$ is the only AP fixed point 
satisfying  $\|x_*\|= \|b\|$. 
\end{prop}

While we do not have the assurance of a unique  AP fixed point in comparison with DR,
 AP has a better convergence rate than DR as we discuss next. 

\subsection{Local convergence}

\begin{thm} \label{thm1} 
 Under the assumptions of Corollary \ref{cor:u} including \eqref{connect},  let $A^*$ be the measurement matrix and  \beq
  B:=A\,\,\diag\lt\{\sgn({A^*f})\rt\} \in \IC^{N\times M}.\label{B}
 \eeq
Then \beq\label{lam2}
 \gamma=\max\{ \|\Im(B^* u)\|: u\in \IC^N,~ u\perp \im f,  ~\|u\|=1\}<1. 
 \eeq
Moreover, for any given $0<\ep<1-\gamma$, if the initial point $y^{(1)}$ is chosen such that 
\[
\|\al^{(1)} x^{(1)}-f\|:=\min_{\al \in \IC \atop |\al|=1}\|\al x^{(1)}-f\|\quad\mbox{ is sufficiently small,}
\]
 then  we have the geometric convergence 
\beq\label{23}
\mbox{\bf \rm DR:}& \min_{\al \in \IC \atop |\al|=1} \| \alpha x^{(k)}-f\|\le (\gamma+\ep)^{k-1} \| \alpha^{(1)} x^{(1)}-f\|,\\
\mbox{\bf \rm AP:}& \label{41'} \min_{\al \in \IC \atop |\al|=1} 
{\|\alpha x^{(k)}- f\|}\le (\gamma^2+\ep)^{k-1} {\|\alpha^{(1)}x^{(1)}- f\|}. 
\eeq
\end{thm}
For $\gamma<1$, the  convergence rate $\gamma^2$ of AP is better than the convergence rate $\gamma$ of DR. 
The proof is omitted as can be adapted to the ptychographic setting from the nonptychographic setting 
of \cite{AP-phasing, DR-phasing} without major changes. However, we will elaborate on the meaning of  and
give an estimate for  \eqref{lam2} below.

First let us explain the connection between the matrix $B$ in \eqref{B} and the subdifferential of the iterative map.
To this end, we consider the isomorphism $\IC^N\cong \IR^{2N}$ via the map
\[
G(x):=\lt[\begin{matrix} \Re(x)\\
\Im(x)\end{matrix}\rt],  \quad G(-ix)=\lt[\begin{matrix} \Im(x)\\
-\Re(x)
\end{matrix}\rt]
\]
and define 
the real-valued matrix
\beq
\mathcal{B}&=&\left[
\begin{array}{c}
\Re(B)     \\
 \Im(B)
\end{array}
\right]\in \IR^{2N\times M}.
\eeq
Denote the AP map by 
\[
F_{\rm AP}=P_1P_2
\]
and the DR map by 
\[
F_{\rm DR}=I+P_1 (2P_2 - I)- P_2. 
\]
From straightforward but somewhat tedious algebra, we have 
\beq\label{ap1}
G(dF_{\rm AP}(f)\xi)=G(\im B\Im(B^* \xi)),\quad \forall\xi\in \IC^N
\eeq
or equivalently
\beq\label{ap2}
G(-\im dF_{\rm AP}(f)\xi)=\cB\cB^\top G(-\im \xi), \quad \forall\xi\in \IC^N
\eeq
and
\beq\label{dr1}
dF_{\rm DR}(f)\eta=\diag[\sgn(A^*f)]J \diag[\overline{\sgn(A^*f)}] \eta
\eeq
\mbox{where}
\beq
\label{dr2}
Jy=(I-B^*B)\Re(y)+\im B^*B \Im(y). 
\eeq
Eq. \eqref{ap1}-\eqref{dr2} exhibit the central role of  $B$ in the subdifferentials $dF_{\rm AP}, dF_{\rm DR}$  at the point $f$.
For detailed derivation we refer the reader to \cite{AP-phasing, DR-phasing}.

Next we explain the meaning of the variational principle \eqref{lam2}.  Let $\lambda_1\ge \lambda_2\ge \ldots\ge \lambda_{2N}\ge 0$ be the singular values of 
\[
\mathcal{B}^\top=\left[
\Re(B)^\top\quad
 \Im(B)^\top
\right]\in \IR^{M\times 2N}.
\]
Since   the complex matrix $B^*$ is isometric, we
have $\lamb_{k}^2+\lamb_{2N+1-k}^2=1, \forall k=1,\ldots,2N$. 

By definition,  for any $x\in \IC^N$
 \beq\nn
&& B^* x=  \diag\lt[\overline{\sgn({A^* f})}\rt] A^* x\eeq
and hence 
\beq
&&\cB^\top G(f)= \Re[B^* f]=|A^* f|.  \label{56}\eeq
On the other hand, we have by isometry of $A^*$
\beq
\label{45}
\cB|A^* f|= \lt[\begin{matrix}
\Re(B|A^*f|)\\
\Im(B|A^*f|)
\end{matrix}\rt]=\lt[\begin{matrix}
\Re(AA^*f)\\
\Im(AA^*f)
\end{matrix}\rt]=\lt[\begin{matrix}
\Re(f)\\
\Im(f)
\end{matrix}\rt]=G(f).
\eeq
Eq. \eqref{56} and \eqref{45} imply $\lambda_1=1$ and $G(f)$ is a leading singular vector of $\cB^\top$. 
We can also easily verify
\beq
\label{last}
 \cB^\top G(-\im f)=\Im[B^*f]=0
 \eeq
 and hence $G(-\im f)$ is  a corresponding singular vector to $\lambda_{2N}=0$. 

Note again $
\Im[B^* u]=\cB^\top G(-\im u).$
The orthogonality condition $\im u\perp f$ is equivalent to $
G(f)\perp G(-\im u).$
Therefore 
$\gamma$ defined in   \eqref{lam2}  is the second largest singular value $\lamb_2$ of $\cB^\top$ 
and admits the variational principle
 \beq 
\label{63} \gamma
&=&\max \{\|\cB^\top u\|: {u\in \IR^{2N},  u\perp G(f)}, \|u\|=1\}. 
 \eeq
It is now straightforward to verify that the two variational principles, \eqref{lam2} and \eqref{63}, are equivalent. We will, however, continue to use \eqref{lam2} which is more convenient than \eqref{63}. 

\subsection{Spectral gap}
Finally, how do we see that $\gamma<1$?

From
\beq\Im(B^*x)
=\Im \lt({\overline{A^* f}\over |A^* f|} \odot A^*x\rt)
=\sum_{j=1}^M {\Re(a_j^* f) \Im( a_j^* x)-\Im( a_j^*f) \Re(a_j^* x)\over
(\Re^2(a_j^*f)+\Im^2( a_j^*f))^{1/2}}\label{57'}
\eeq
we have by the Cauchy-Schwartz inequality and the isometry of $A^*$
\beq\label{47}
\|\Im(B^* x)\|^2
&\le &
 \sum_{j=1}^M  \Re^2( a_j^* x)+ \Im^2(a_j^*x)=\sum_{j=1}^M |a_j^*x|^2=\|A^* x\|^2=\|x\|^2.
\eeq

In view of \eqref{57'}, the inequality becomes an equality if and only  if \beq\label{95}
\Re(a_j^* x)\Re(a_j^*f)+\Im( a_j^*x) \Im( a_j^*f)=0, \; \forall j=1,\cdots N,
\eeq
where $a_j$ are the columns of $A$, 
or equivalently
\beq\label{95'}
\sgn\{A^*x\}= \sigma\odot \om_0 
\eeq
where the components of $\sigma$ are either 1 or -1, i.e. $
 \sigma[j]\in \{1, -1\}, \quad \forall j=1,\cdots N. $
 
 Now we recall from \cite{DR-phasing} the following uniqueness theorem for the non-ptychographic setting. 
 \begin{prop}\label{prop:u}  (Uniqueness of Fourier magnitude retrieval)
 Suppose $f$ is not a line object and let the mask $\mu$'s phase be continuously and independently distributed. 
If for  the  matrix \eqref{one} we have 
\beq
\label{mag}
\measuredangle A^*\hat x=\pm \measuredangle A^* f
\eeq
where
 the $\pm$ sign may be  pixel-dependent, then almost surely $\hat x= c f$ for some constant $c\in \IR$. 
\end{prop}
This result implies that from the {\em Fourier phase} data, up to a $\pm$ sign,  for each mask in the ptychographic setting we can identify 
the illuminated part of the object, up to a real constant.  Now for any ptychographic scheme under 
the minimum overlap condition \eqref{two2}, the constants associated with all the masked domains must be the same. Hence \eqref{47} is a strict inequality and  $\gamma<1$.

\section{Convergence rate bound}\label{sec:rate}

In this section, we give an estimate of  $\gamma$ and exhibits an explicit dependence of $\gamma$ on the parameter $q$ of the minimalist scheme introduced in Section \ref{sec:scheme}. 

 We divide the initial mask domain $\cM^0$, now denoted as $\cM^{00}$,  into four equal blocks
$\cM^{00}=\bigcup_{i,j=0}^1 \cM^{00}_{ij}$ or in the matrix form
\[
\cM^{00}=
\left[
\begin{array}{cc}
\cM^{00}_{00}  &  \cM^{00}_{10}   \\
\cM^{00}_{01}  &\cM^{00}_{11} 
\end{array}
\right],\quad \cM^{00}_{ij}\in \IC^{m/2\times m/2},\quad i,j=0,1.
\]
 For $m=2n/q$, let $T^{kl}= {m\over 2}({k}, {l})$. Denoting the $T^{kl}$-shift of $\cM^{00}$ by $\cM^{kl}$  we have $
\cM^{kl}=\bigcup_{j,k=0}^1 \cM_{ij}^{kl} $ where  $\cM_{ij}^{kl}$ is the $T^{kl}$-shift of $\cM^{00}_{ij}$
 The corresponding 
partition of the initial mask $\mu^{00}$ and  $T^{kl}$-shifted mask $\mu^{kl}$ can be written as
\[
 \mu^{00}=
\left[
\begin{array}{cc}
\mu^{00}_{00}  &  \mu^{00}_{10} \\
\mu^{00}_{01} &\mu^{00}_{11}
\end{array}
\right],\quad  \mu^{kl}=
\left[
\begin{array}{cc}
\mu^{kl}_{00}  &  \mu^{kl}_{10}   \\
\mu^{kl}_{01}  &\mu^{kl}_{11} 
\end{array}
\right].
\]

For convenience, we consider the periodic boundary condition on the whole object domain, i.e.
\beq
\label{bc1}
\cM^{q-1,l}_{10}=\cM^{0l}_{00}, &&\quad \cM^{q-1,l}_{11}=\cM^{0l}_{01}, \\
\cM^{k,q-1}_{01}=\cM^{k0}_{00}, &&\quad \cM^{k,q-1}_{11}=\cM^{k0}_{10},  \label{bc2}\\
\mu^{q-1,l}_{10}= \mu^{0l}_{00},&&\quad \mu^{q-1,l}_{11}=\mu^{0l}_{01},\\
\mu^{k,q-1}_{01}= \mu^{k0}_{00},&&\quad \mu^{k,q-1}_{11}=\mu^{k0}_{10},
\eeq
for all $j,k=1,...,q-1$.

Accordingly, we divide the object $f$ into $q^2$ non-overlapping blocks
  \beq \label{x,g}
 f=
 \commentout{\left(
\begin{array}{ccc}
f_1  \\
\vdots \\
f_{q}
\end{array}
\right)=
}
\left[
\begin{array}{ccc}
f_{11}  & \ldots  & f_{1q}   \\
\vdots  & \vdots  &\vdots   \\
f_{q1}  &\ldots   &  f_{qq}  
\end{array}
\right],\quad f_{ij}\in \IC^{m/2\times m/2}. \eeq

Let the ODFT $\Phi^{kl}$ defined on $\cM^{kl}$ be divided  into four equal blocks
\[ \Phi^{kl}=
\left[
\begin{array}{cc}
\Phi^{kl}_{00}  &  \Phi^{kl}_{10} \\
\Phi^{kl}_{01} &\Phi^{kl}_{11}
\end{array}
\right]
\]
where each $\Phi^{kl}_{ij}:  \IC^{m/2 \times m/2} \to \IC^{2m \times 2m}$ is a rank-3 tensor defined on $\cM^{kl}_{ij}$ and normalized such that
     \beq \label{FC}
\Phi^{kl*}_{ij} \Phi^{kl}_{i'j'}=\frac{\delta_{ij,i'j'}}{4}I_{m/2\times m/2}, \quad  i,j,i',j'=0,1. 
\eeq
Analogous to \eqref{x,g} the diffracted field  $h=A^*f$ can be partitioned into $q\times q$ blocks, $[h_{kl}]$,
where
\[ 
h_{kl}=\sum_{i,j=0}^1
\Phi^{kl}_{ij}(\mu^{kl}_{ij}\odot f_{i+k,j+l}), \quad k,l =1,\ldots, q
\]
where $f_{i+k,j+l}$ is cyclically defined with respect to the subscript.

\begin{prop}\label{prop:gap}
For the minimalist scheme, $\gamma$ defined in \eqref{lam2} satisfies 
\beq
\label{bound}
\gamma> 1-C/q^2
\eeq
 for some constant $C$ depending on $f$, but independent of $q$. 
 \begin{rmk}
 Note that the derivation of the bound \eqref{bound} does not assume a random mask and
 is valid for the minimalist scheme with any mask. 
 \end{rmk}
\commentout{
Then 
\[
\lambda_{2N-1}\le \sqrt{2}
\frac{\max_{j=1,\ldots, q} \|f_j\|}{\min_{j=1,\ldots, q} \|f_j\|_{\rm  F}} \sin \frac{\pi}{q}.
\]
Hence, according to $\lambda_2^2+\lambda_{2n-1}^2=1$, $\lambda_{2}\ge  1-c/q^2$ for some constant $c>0$, depending on $f$, but independent of $q$. 
}
\end{prop}

\begin{proof}
For simplicity, we assume $\|f\|=1$.  Analogous to \eqref{lam2}, we have
from \eqref{last} the variational principle 
 \beq \label{costF}
\lambda_{2N-1}=\min\{\|\Im( B^* g)\|: g\in \IC^{n\times n}, g\perp f,  \|g\|=1\}
\eeq

Denote $\om=\sgn(A^* f)$ and 
consider the test function for \eqref{costF}
\[
 g=
\left[
\begin{array}{ccc}
v_1 f_{11}  & \ldots  & v_1f_{1q}   \\
\vdots  & \vdots  &\vdots   \\
v_q f_{q1}  &\ldots   &  v_q f_{qq}  
\end{array}
\right] \in \IC^{n\times n}
\]
where
  \beq \label{v}
  v_j= a \sin (\frac{2\pi j}{q}-c),\quad  j=1,\ldots, q,\eeq 
for some real constants $a, c$ to be selected.

Let $f_j$ be  the $j$-th  row or column of  \eqref{x,g}. The orthogonality condition  $ g\perp f$  leads to 
\beq \label{ortho2}
0=\sum_{i=1}^q v_{i} \sum_{j=1}^q\|f_{ij}\|^2= \sum_{i=1}^q v_{i}  \|f_{i}\|^2=
a\sum_{j=1}^q
\|f_j\|^2 \sin(2\pi j/q-c)
=:p(c).
\eeq 
That is,
 $c$ needs to be a real root of 
$p$. Since $ p(0)=-p(\pi)$, the existence of a root $c\in [0,\pi]$ follows from the intermediate value theorem. 
On the other hand, to satisfy 
$\|g\|=1$ we need \beq \label{norm2}
1=\sum_{j=1}^q v_{j}^2 \|f_{j}\|^2=a^2 \sum_{j=1}^q \sin^2 ({2\pi j}/{q}-c)\|f_j\|^2
\eeq 
which implies
\beq
\label{a2}
a^2= \lt(\sum_{j=1}^q \sin^2 ({2\pi j}/{q}-c)\|f_j\|^2\rt)^{-1}. 
\eeq

Write  \[
 A^* g=
\left(
\begin{array}{ccc}
h_{11}  & \ldots  & h_{1q}   \\
\vdots  & \vdots  &\vdots   \\
h_{q1}  &\ldots   &  h_{qq}  
\end{array}
\right),
\]
where 
\[ 
h_{kl}=\sum_{i,j=0}^1
\Phi^{kl}_{ij}(\mu^{kl}_{ij}\odot f_{i+k,j+l})v_{k+j}, \quad k,l =1,\ldots, q. 
\]
Likewise, we have
   \[
\om=\sgn({A^* f})
=\left(
\begin{array}{ccc}
\om_{11}  & \ldots  & \om_{1q}   \\
\vdots  & \vdots  &\vdots   \\
\om_{q1}  &\ldots   &  \om_{qq}  
\end{array}
\right),\quad
\om_{kl}=\sgn\lt\{
h_{kl}=\sum_{i,j=0}^1
\Phi^{kl}_{ij}(\mu^{kl}_{ij}\odot f_{i+k,j+l})\rt\}. 
 \] 
To calculate $\Im(B^* g)=\Im(\overline{\om}\odot A^* g)$, we introduce
\beq
u_{ij}&=&\Im\left[
\overline{ \om}_{ij}
\odot (\Phi^{ij}_{00}\mu^{ij}_{00}\odot f_{ij} +\Phi^{ij}_{10} \mu^{ij}_{10}\odot f_{i+1,j})\right]\in\IR^{4n/q\times 4n/q}\\
u'_{ij}&=&\Im\left[
\overline{ \om}_{ij}
\odot (\Phi^{ij}_{01} \mu^{ij}_{01}\odot f_{i,j+1} +\Phi^{ij}_{11} \mu^{ij}_{11}\odot f_{i+1,j+1})\right]\in\IR^{4n/q\times 4n/q}.
\eeq
Note that 
 \[u_{ij}+u'_{ij}=\Im\left[
\overline{ \om}_{ij}\odot
\sum_{k,l=0}^1\Phi^{ij}_{kl}( \mu_{kl}^{ij}\odot f_{i+k,j+l})\right] =
\Im\lt|\sum_{k,l=0}^1\Phi^{ij}_{kl}( \mu_{kl}^{ij}\odot f_{i+k,j+l})\rt|=0\]
and hence 
\beq \label{w}
\|\Im(\overline{\om}_{ij}\odot h_{ij})\|^2=\|
u_{ij} v_{i}+u_{ij}' v_{i+1}\|^2=\|u_{ij} \|^2 (v_{i}-v_{i+1})^2.
\eeq 

Let $c_{ij}^{kl}$ be the norm of the mapping $
f_{ij}\longrightarrow \Im(\overline{\om}_{ij}\odot F_{kl}^* (\mu^{ij}_{kl}\odot f_{i+k,j+l}))$. 
By (\ref{FC}) 
we have $c_{ij}^{kl}\in [0,1/2]$. Thus,
 \beq \label{mu}
\sum_{j=1}^q \| u_{ij}\|^2\le  \sum_{j=1}^q ( (c^{00}_{ij})^2\| f_{ij}\|^2 +(c^{10}_{ij})^2\| f_{i+1,j}\|^2)\le   \frac{1}{2} \|f_{i}\|^2.
\eeq 

Combining 
 (\ref{w}), (\ref{v}), (\ref{norm2}), \eqref{a2} and (\ref{mu}) we have the following calculation 
 \beq
\label{53}\lamb_{2N-1}^2&\le &\|\Im(\overline{\om}\odot (A^* g))\|^2\\
&=&
\sum_{i=1}^{q} \sum_{j=1}^q 
\|\Im(\overline{\om}_{ij}\odot h_{ij})\|^2\nn\\
&=&\sum_{i=1}^{q} (v_{i}-v_{i+1})^2 \sum_{j=1}^q \| u_{ij}\|^2\nn\\
&\le& {a^2\over 2}  \sum_{i=1}^q \|f_i\|^2 \left|\sin ({2\pi i}/{q}-c)-\sin ({2\pi (i+1)}/{q}-c)\right|^2\nn\\
&\le& {\sum_{i=1}^q \|f_i\|^2 \left|\sin ({2\pi i}/{q}-c)-\sin ({2\pi (i+1)}/{q}-c)\right|^2\over
2\sum_{j=1}^q \sin^2 ({2\pi j}/{q}-c)\|f_j\|^2} \nn\\
&=&  4\sin^2({\pi}/{q}) \frac{\sum_{j=1}^q  \|f_j\|^2 \cos^2({2 \pi (j+0.5)}/{q}-c)}{
2\sum_{j=1}^q \|f_j\|^2\sin^2 ({2\pi j}/{q}-c)}.\nn
\eeq
Note that 
\beqn
\frac{\sum_{j=1}^q \cos^2({2 \pi (j+0.5)}/{q}-c)  \|f_j\|^2}{
2\sum_{j=1}^q \sin^2 ({2\pi j}/{q}-c)\|f_j\|^2
}
&\le& \frac{\max_j \|f_j\|^2\sum_{j=1}^q  \cos^2({2 \pi j}/{q}-c+{\pi}/{q}) }{2 \min_j  \|f_j\|^2
\sum_{j=1}^q \sin^2 ({2\pi j}/{q}-c)}\\
&\le& \frac{ \max_j  \|f_j\|^2}{2 \min_{j}  \|f_j\|^2}:=c
\eeqn
and hence
\[
\lamb_{2N-1}^2\le 4c\sin^2({\pi}/{q}). 
\]
The desired result then follows
from the identify $\lamb_2^2+\lamb_{2N-1}^2=1$. 
\end{proof}

\section{Numerical experiments}\label{sec:num}
\begin{figure}[h]
\centering
\subfigure[Modulus of RPP]{\includegraphics[width=5cm]{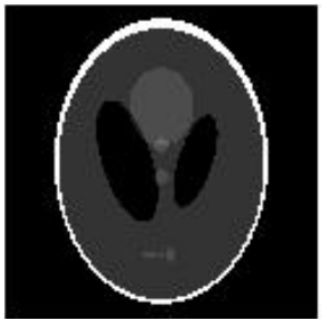}}
        \subfigure[$\rho={3\over 25\pi}\approx 0.038$]{\includegraphics[width=4.8cm]{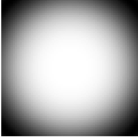}}
            \subfigure[$\rho={6\over 5\pi}\approx 0.38$]{\includegraphics[width=4.8cm]{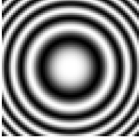}} 
                 \subfigure[$\ell=16, m=64$]{\includegraphics[width=5cm]{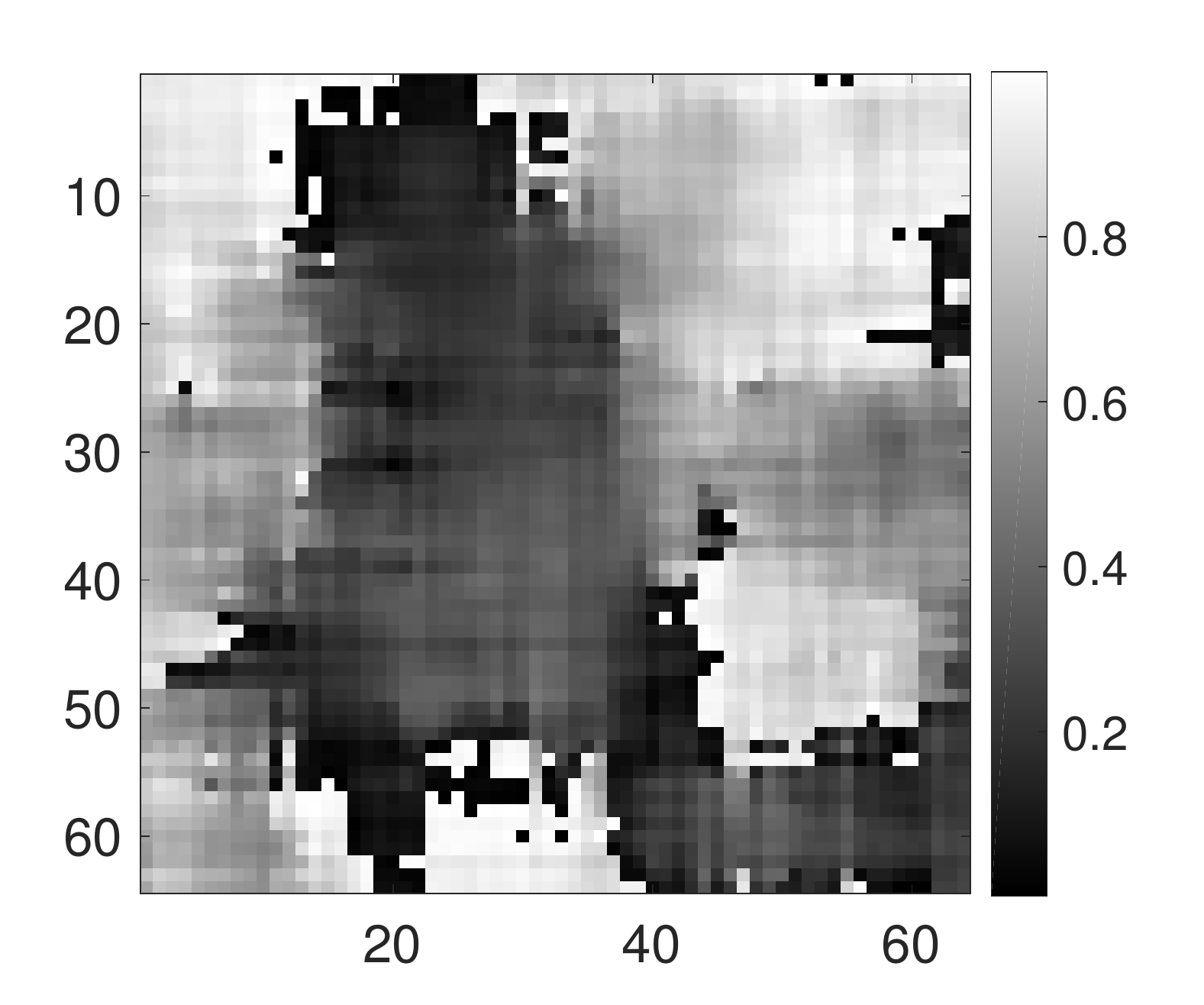}}
   \subfigure[$\ell=8, m=32$]{\includegraphics[width=5cm]{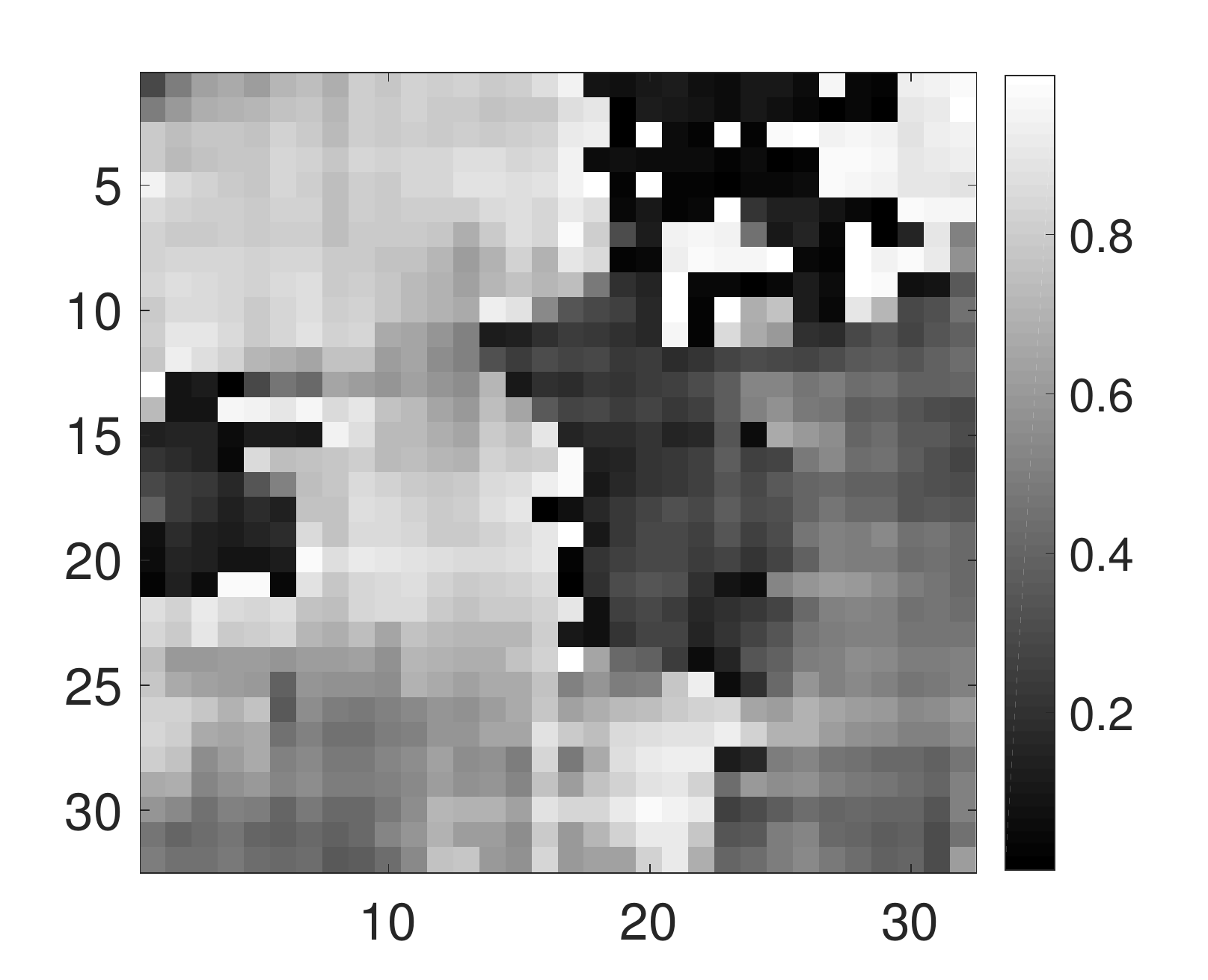}}
    \subfigure[$\ell=4, m=16$]{\includegraphics[width=5cm]{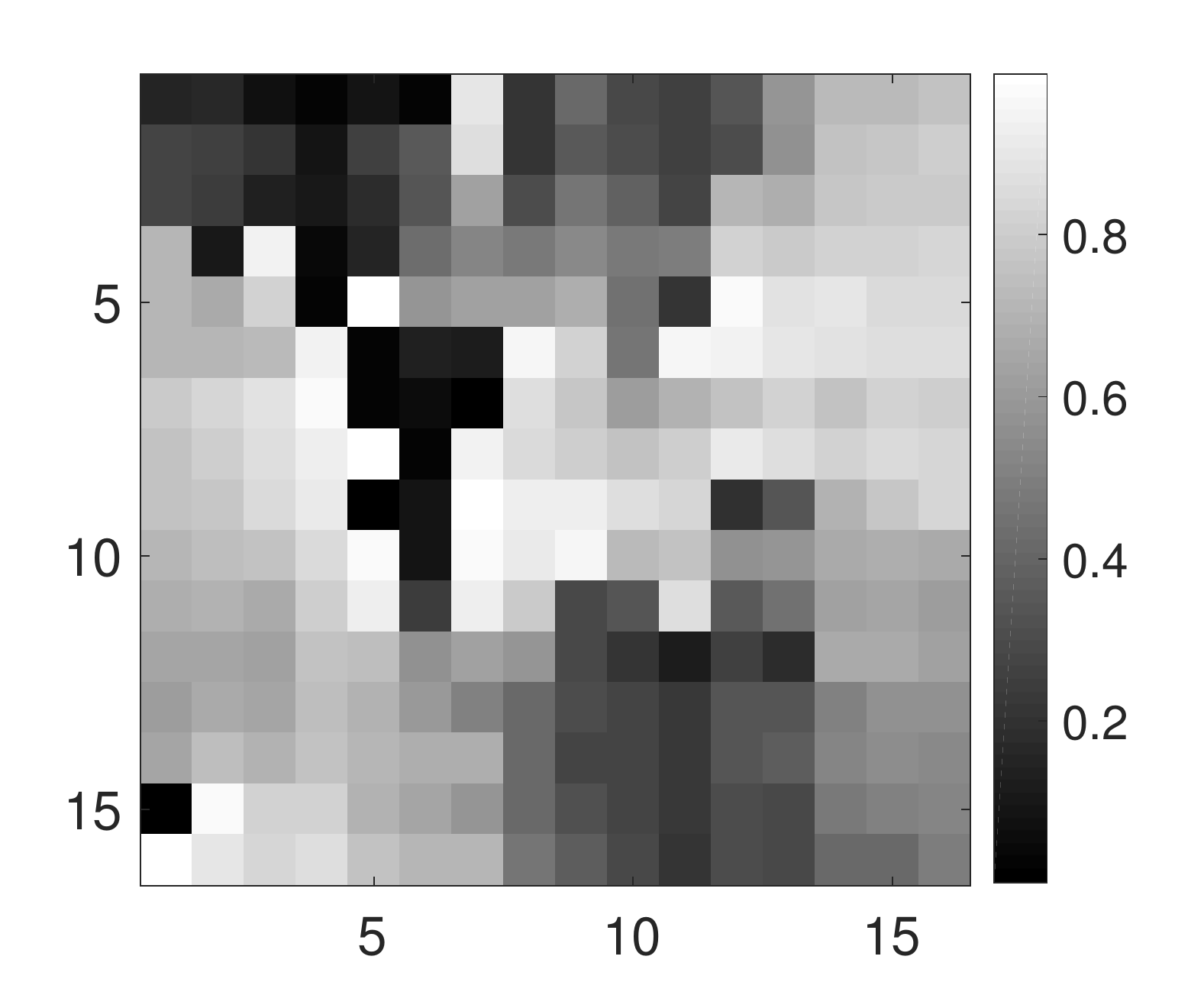}}     
 \caption{(a) The modulus of $128\times 128$ RPP;  (b)\&(c) The modulus of the real part of the respective Fresnel masks; 
 (d), (e)\&(f) The phases of various correlated random masks  in the unit of $2\pi$.} 
 \label{fig:mask}
\end{figure}

A primary purpose of our numerical experiments
is to find out how $q$ affects the numerical reconstruction and propose a practical guideline for
using the minimalist scheme. 
We also want to see how the complexities of the mask and the object affect numerical performance. 
Finally, we want to test how robust the minimalist scheme is with respect to measurement noise. 

As pointed out above, DR has the true solution as the unique fixed point in the object domain  (Proposition \ref{prop4}) while AP has a better convergence rate than DR (Theorem \ref{thm1}). A natural way to combine their
strengths is to use DR as the initialization method for AP. 
We choose AP and DR as the building blocks of
our reconstruction algorithm because of the fixed point and convergence properties 
and also because there are no adjustable parameters which can be tuned to optimize the performance as in other algorithms \cite{ADM, Hesse}. We do not claim that this combination  
yields the best algorithm for ptychography.  Quite the contrary, our approach can be
easily improved,  for example, by initializing AP with  
the DR iterate of the {\em least residual} within a given number of iterations
instead of the last iterate. 

Our test image $f$ is {randomly phased phantom (RPP)}:
the phantom (Fig. \ref{fig0} (a)) with phase at each pixel being independent and  uniformly distributed over a specific range, referred to as the angle range hereafter. RPP is chosen for two reasons: (i) the core image is surrounded by dark pixels and
the  loose support  makes RPP more challenging  to reconstruct than an image of a tight support;
(ii) the adjustable angle range is a convenient way for controlling 
the object complexity.

We use the relative error  (RE) and residual (RR) as figures of merit
for the recovered image $\hat{f}$:  
\beqn
\mbox{\rm RE} (\hat{f}) &= &
 \min_{\alpha\in \IR}\frac{\|f-e^{i\alpha}\hat{f}\|}{\|f\|}
\nn\\
\mbox{\rm RR}(\hat{f}) &= &\frac{\| \ b - |A^*\hat{f}| \ \| }{\| b\|}. 
\eeqn

\subsection{Random and Fresnel masks}

We consider two kinds of 
random masks
$e^{\im\theta(\bn)}$ where $\theta(\bn)$ are either  independent, identically distributed (i.i.d.) 
or 
 $\ell$-correlated uniform random variables on $[0,2\pi]$, where $\ell\in \IN$ is the correlation length. 

The correlated random mask is produced by convolving the  i.i.d.  mask with the characteristic function 
of the set $\{(k_1,k_2)\in \IZ^2: |\max\{|k_1|, |k_2|\}|\le\ell/2\}$
\commentout{
\[
g_\ell(k_1,k_2) =\left\{ \begin{matrix}
1, & |\max\{|k_1|, |k_2|\}|\le\ell/2\\
  0,&\hbox{else} 
         \end{matrix}\right.
         \]
\[
g_\ell(k_1,k_2) =\left\{ \begin{matrix}
\exp{[-\ell^2 / (\ell^2 - |\max\{|k_1|, |k_2|\}|^2)]}, & |\max\{|k_1|, |k_2|\}|<\ell\\
  0,&\hbox{else} 
         \end{matrix}\right.
         \]
for $k_1,k_2\in \IZ$} and 
 normalizing pixel-by-pixel to get a phase mask.  The i.i.d. mask corresponds to $\ell= 1$.  

 We also consider 
the Fresnel mask with
\beq\label{par}
\mu^0(k_1,k_2):=\exp\lt\{\im \pi \rho ((k_1-\beta_1)^2+(k_2-\beta_2)^2)/m\rt\},\quad k_1, k_2=1,\cdots, m ~(={2n/q})
\eeq
where $\rho,\beta_1,\beta_2\in \IR$ are adjustable parameters, as well as the plain mask ($\rho=0$). The choice
of $\beta_1, \beta_2$ has an insignificant effect on numerical reconstruction. The form of the discrete Fresnel phase is dictated
by our goal of keeping the angular aperture of the illumination fixed independent of 
$m$ since \eqref{par} describes a point-source illumination
with both the aperture (i.e. the linear size of the mask) and 
 the distance to the object proportional to the parameter $m$ \cite{BW}. 
 If we set the distance from the point source to the object to be $m L$ and the pixel size to be $\delta\times \delta$, then in the Fresnel kernel $\rho=\delta^2/(\lambda L)$ where $\lambda $ is the wavelength. 
 For a different $\rho$, we imagine varying $L$ while keeping $\lambda$ and $\delta$ fixed. 
 The larger $L$ is, the smaller $\rho$ is and hence the coarser the mask is. 
With the minimalist scheme described in Section \ref{sec:scheme} and the Fresnel mask
\eqref{par} with $\rho$ fixed, 
 the total bandwidth of the mask and the total number of measurement data
 are fixed as $q$ varies. Fig. \ref{fig:mask}(b)(c) shows the real part of the Fresnel mask at two different $\rho$. 
 
 In the same spirit, in the case of random mask, we let the correlation length $\ell$ be proportional to
 $m$ as  $q$ varies when we turn to the noisy case. Fig. \ref{fig:mask}(d)-(f) shows three examples of correlated masks with the same ratio $m/\ell=4$. 
 
 \subsection{Twin images with the Fresnel mask}
 Our first experiment has to do with the choice of the Fresnel parameter $\rho$ in \eqref{par}.
 Fig. \ref{fig1} shows that the error spikes around $\rho \in \IN$. This phenomenon is due to the existence of twin-like image. As explained in the appendix, for $q=2$, $g:=Q f\odot Q \mu \odot\overline{\mu}$ produces
 the same ptychographic data as $f$ where $Q$ is the conjugate inversion operation.
 In such a case, ptychographic solutions are not unique and the ambiguity hurts 
 the performance of reconstruction. For $q>2$, the spikes are much smaller than those of $q=2$. 
 As $q$ increases from 2 to 6,  the order of magnitude of fluctuation from peak to valley decreases from more than four orders of magnitude  to about two or less. To avoid the twin-image ambiguity, we 
 choose irrational values of $\rho$. 


  \begin{figure}
\centering
\subfigure[$q=2$]{\includegraphics[width=5cm]{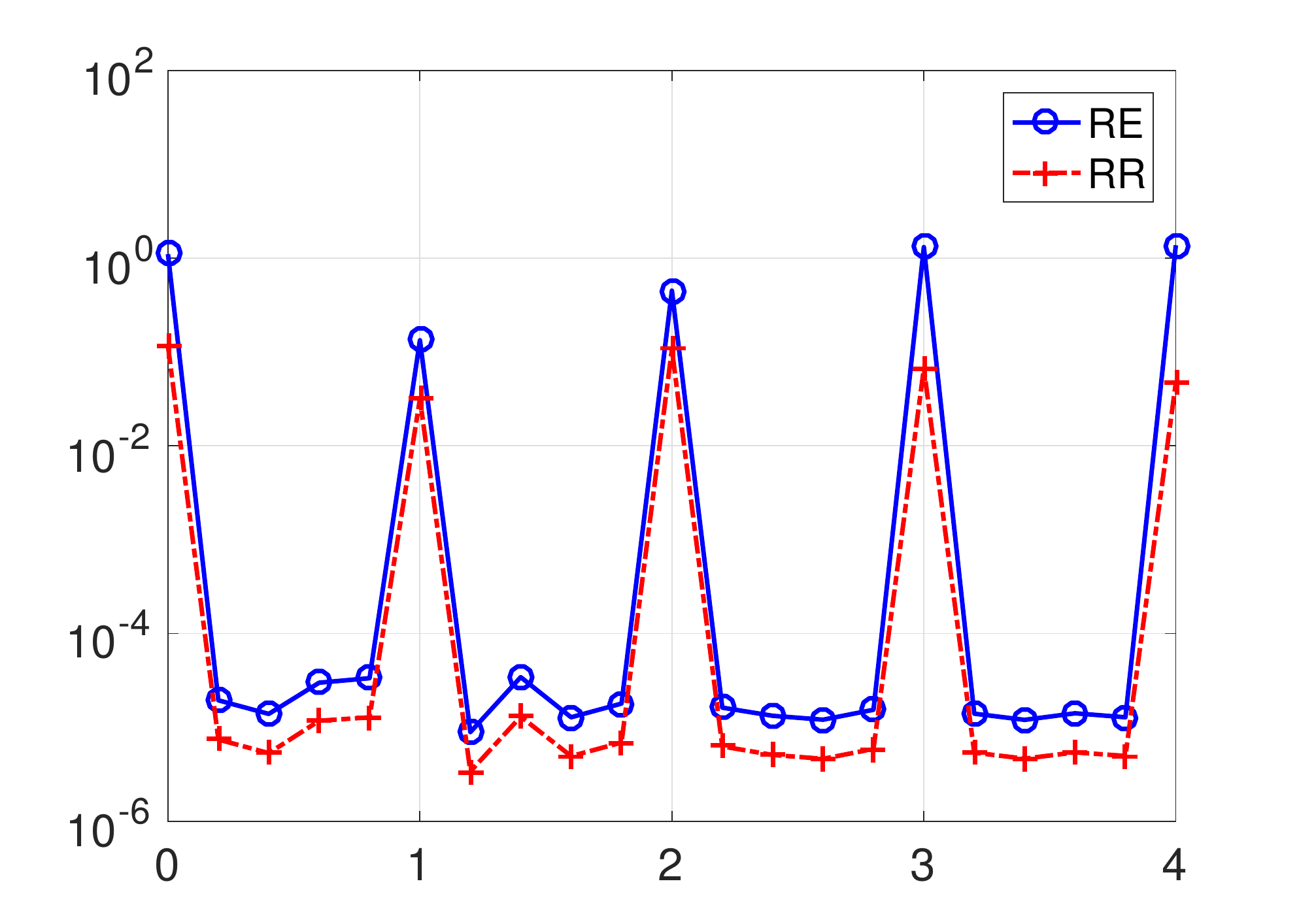}}
 \subfigure[$q=4$]{\includegraphics[width=5cm]{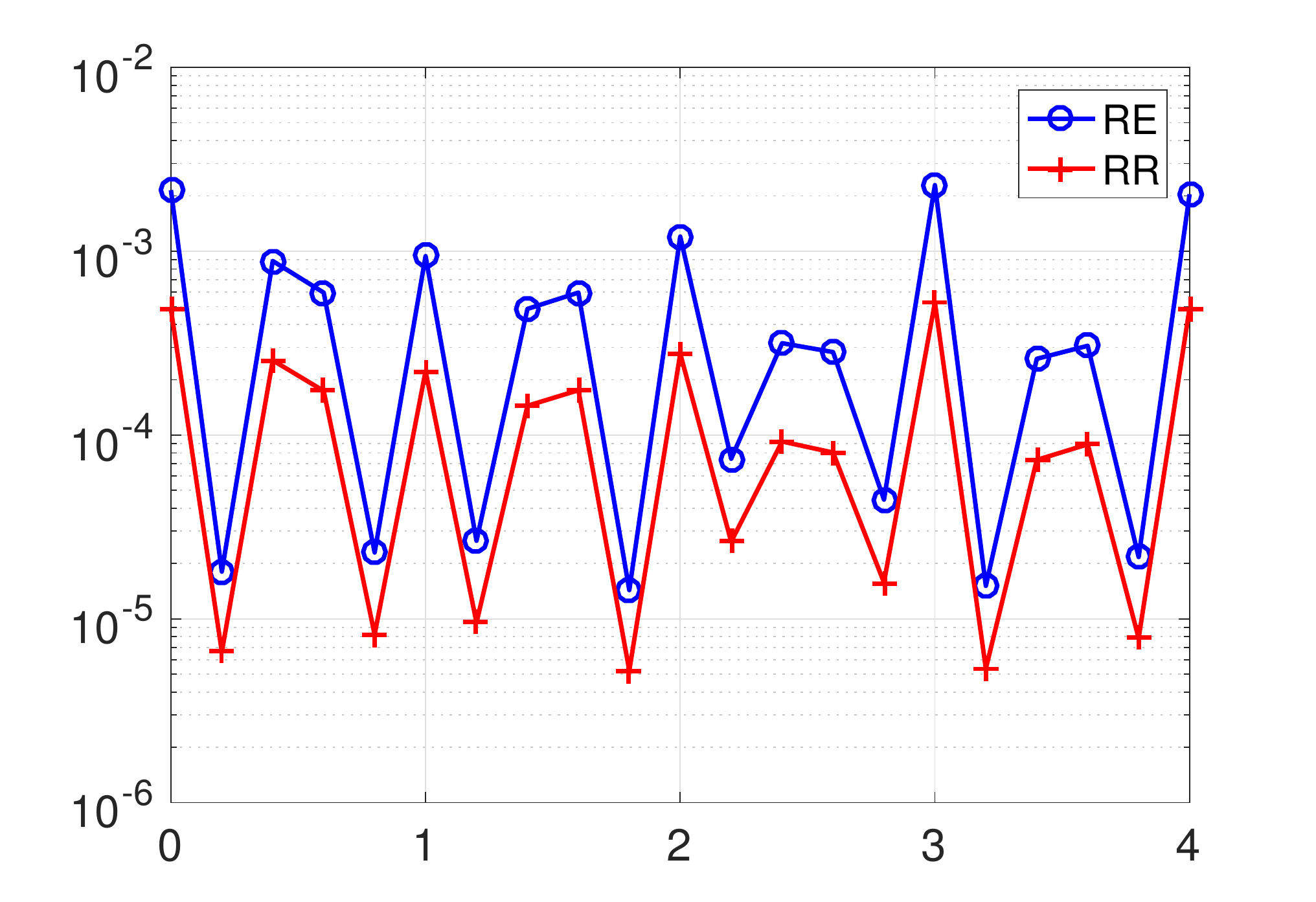}}
  \subfigure[$q=6$]{\includegraphics[width=5cm]{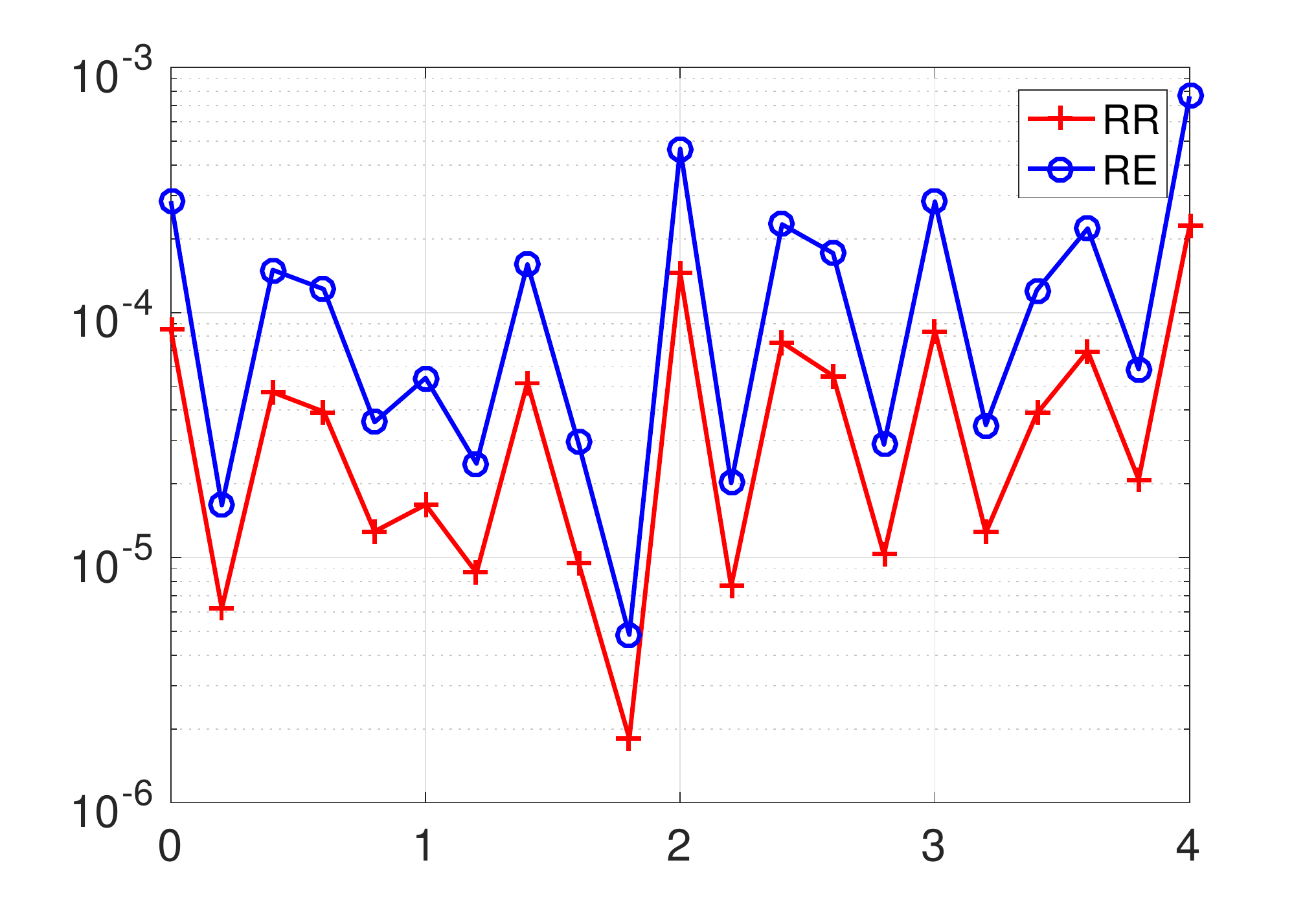}} 
  \caption{RE and RR on the semi-log scale  versus the parameter $\rho$ of the Fresnel mask with 200 FDR iterations followed by 100 AP iterations for $60\times 60$ RPP of the full angle range $ [0, 2\pi]$. 
  }
  \label{fig1}
  \end{figure}

\subsection{Effect of the mask} Fig. \ref{fig:different-object} shows RE versus 100 AP iterations
after the DR initialization 
with the i.i.d. mask and  two Fresnel masks for RPP of various angle ranges (legend). 
We see that the i.i.d. mask produces the best initialization and the fastest convergence rate
and that the Fresnel mask of a larger $\rho$ produces a better initialization and a better convergence rate
than the Fresnel mask of a smaller $\rho$ for all angle ranges. 

\begin{figure}
\centering   
\subfigure[i.i.d. mask]{\includegraphics[width=5cm]{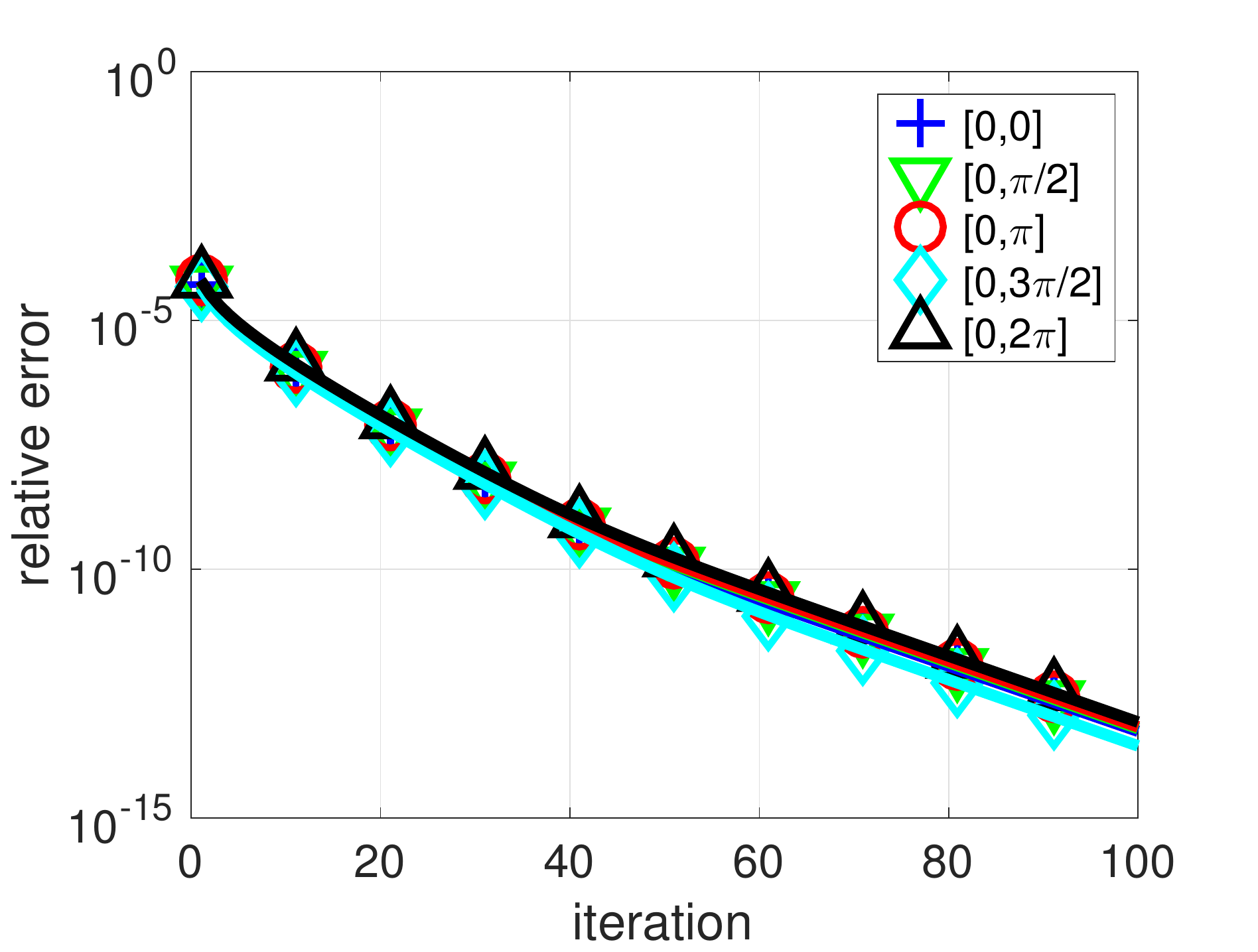}}
\subfigure[Fresnel mask with $\rho= \frac{6}{5\pi}$]{\includegraphics[width=5cm]{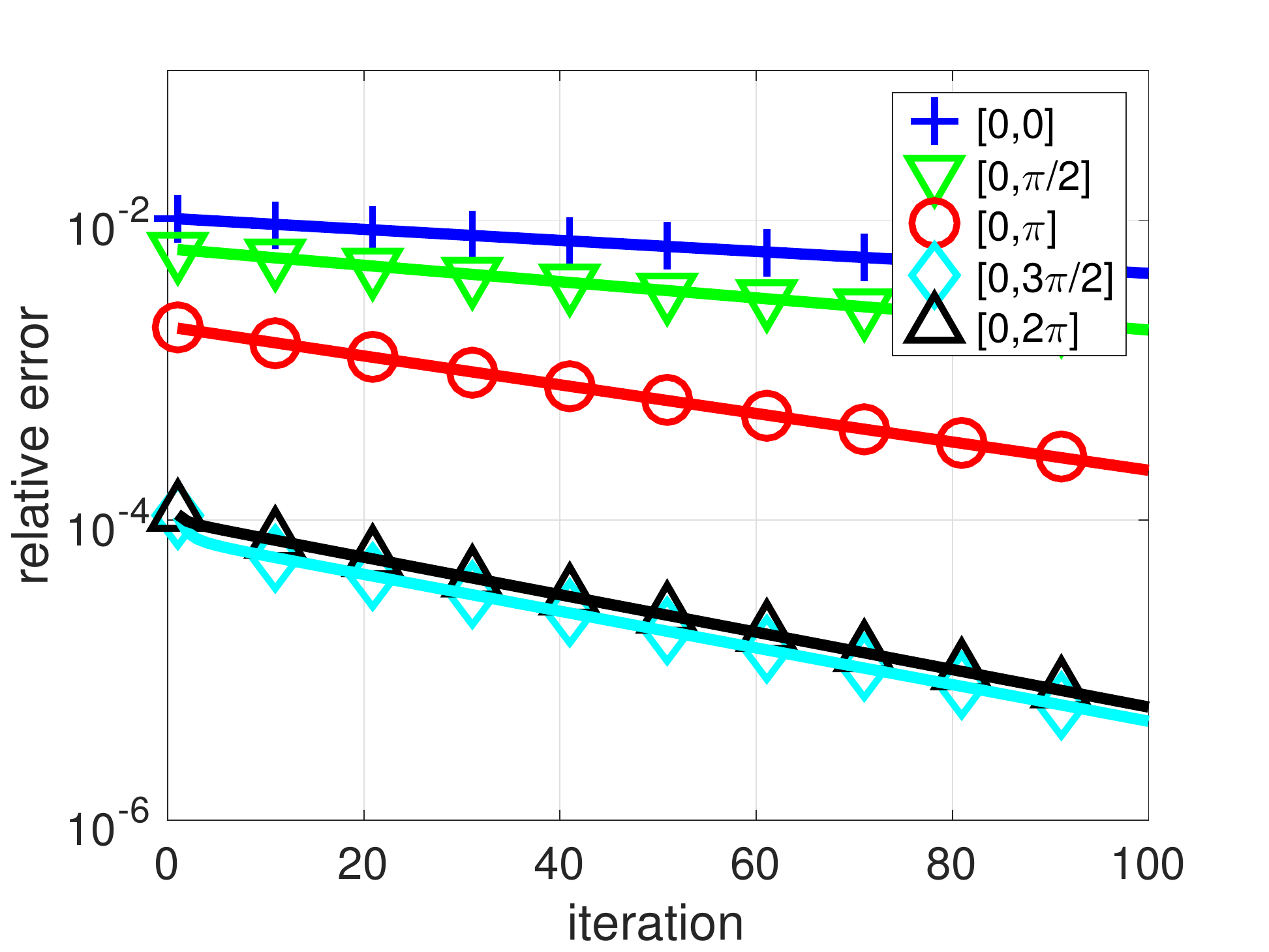}} 
\subfigure[Fresnel mask with $\rho= \frac{3}{25\pi}$]{\includegraphics[width=5cm]{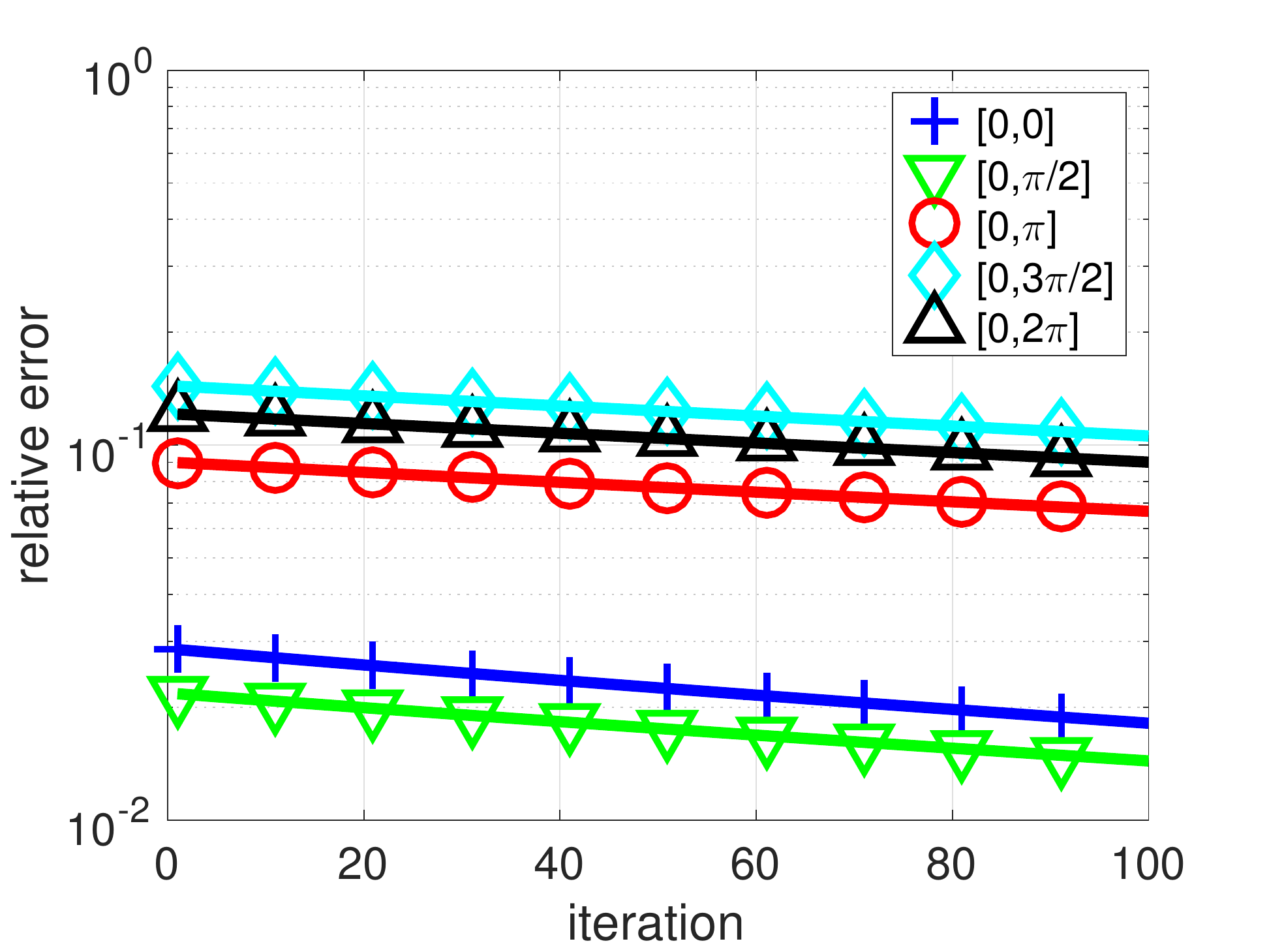}}  
 \caption{ RE on the semi-log scale for $60\times 60$ RPP of various angle ranges (legend) versus  100 AP iterations with initialization given by 300 DR iterations with $q=4$.}
     \label{fig:different-object}
\end{figure}

\subsection{Effect of $q$}

\begin{figure}
\centering
 \subfigure[i.i.d. mask]{\includegraphics[width=5cm]{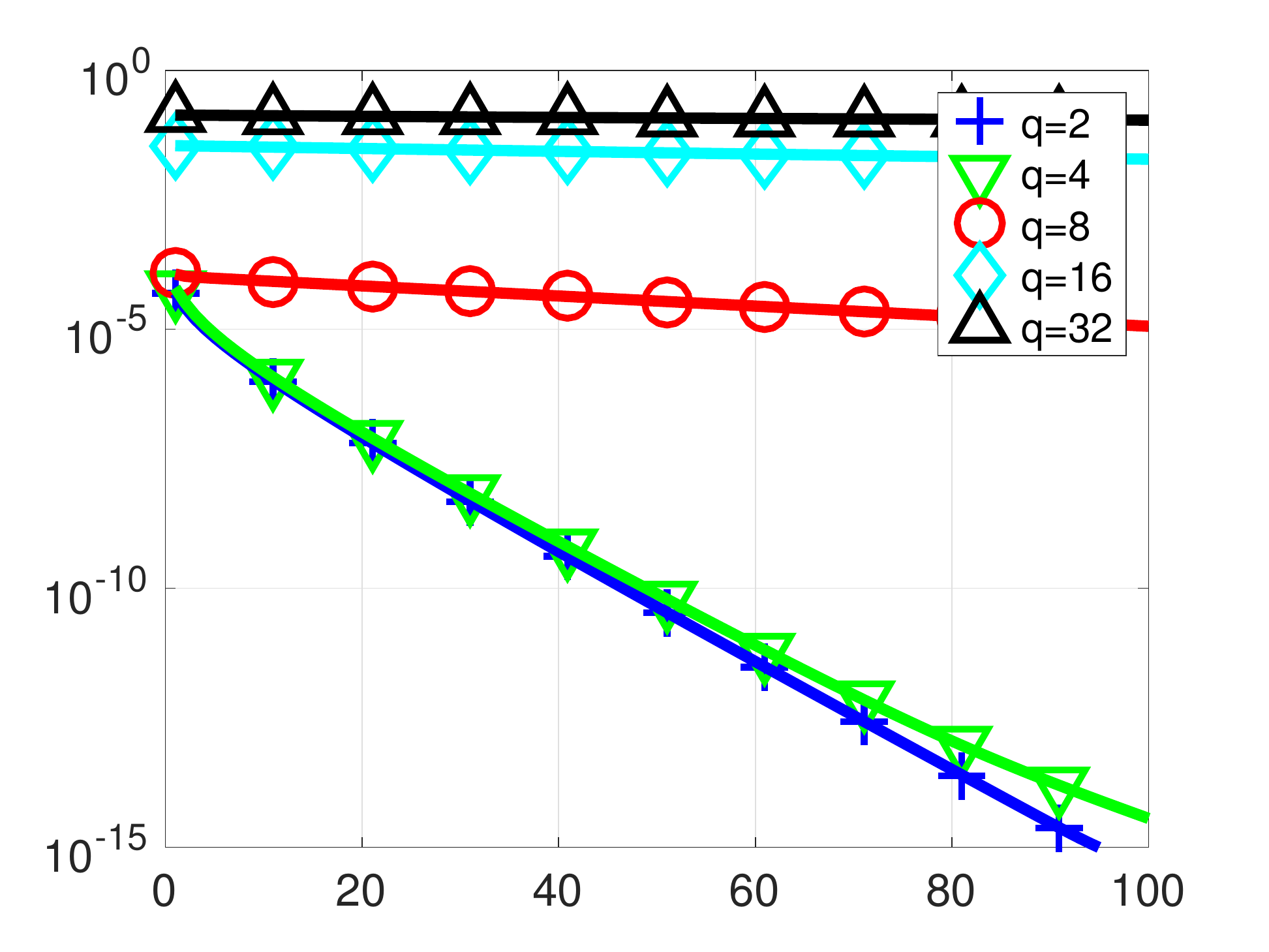}}
  \subfigure[Fresnel mask with $\rho={6\over 5\pi}$]{\includegraphics[width=5cm]{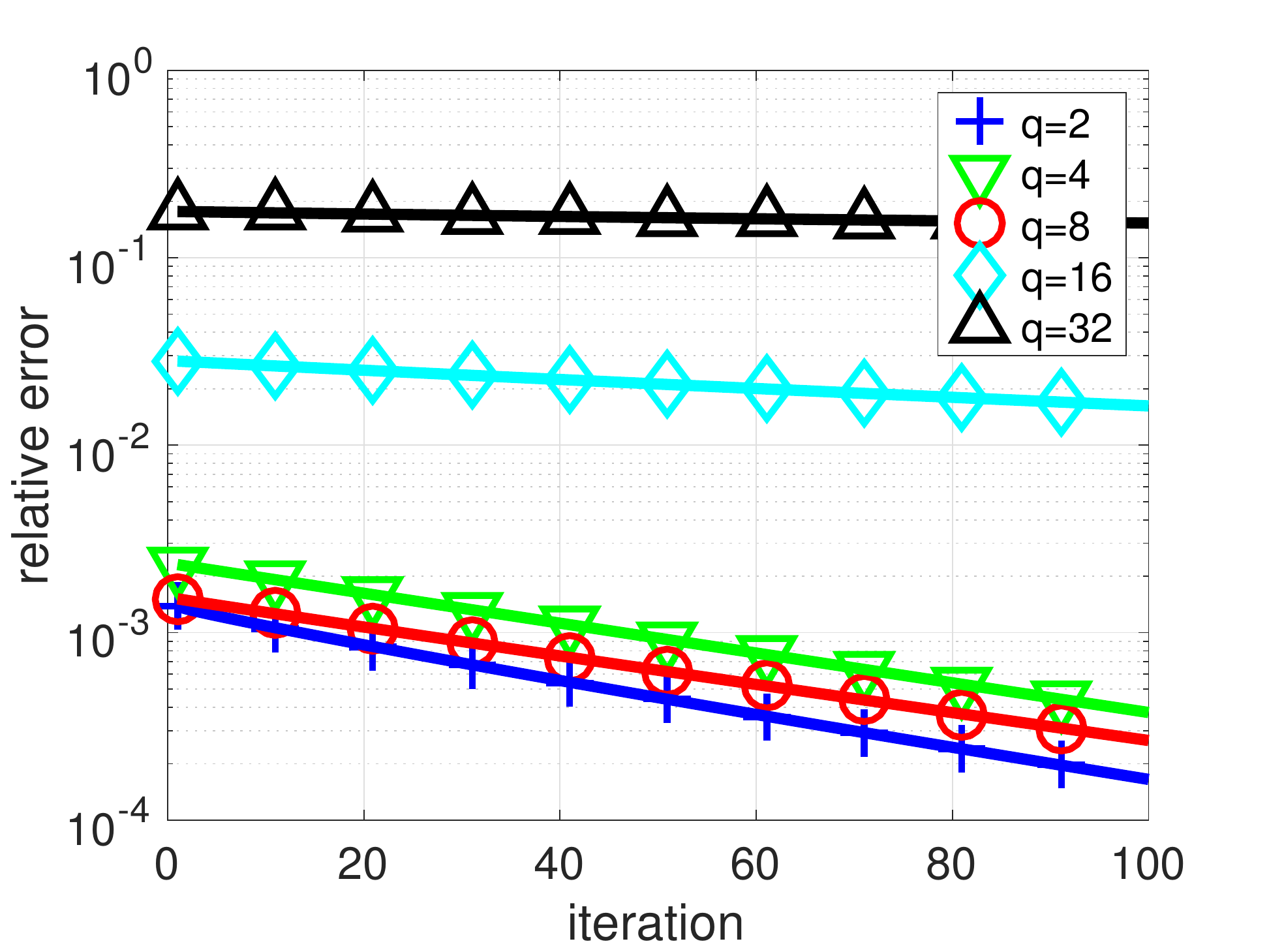}}
  \subfigure[Fresnel mask with $\rho={3\over 25\pi}$]{\includegraphics[width=5cm]{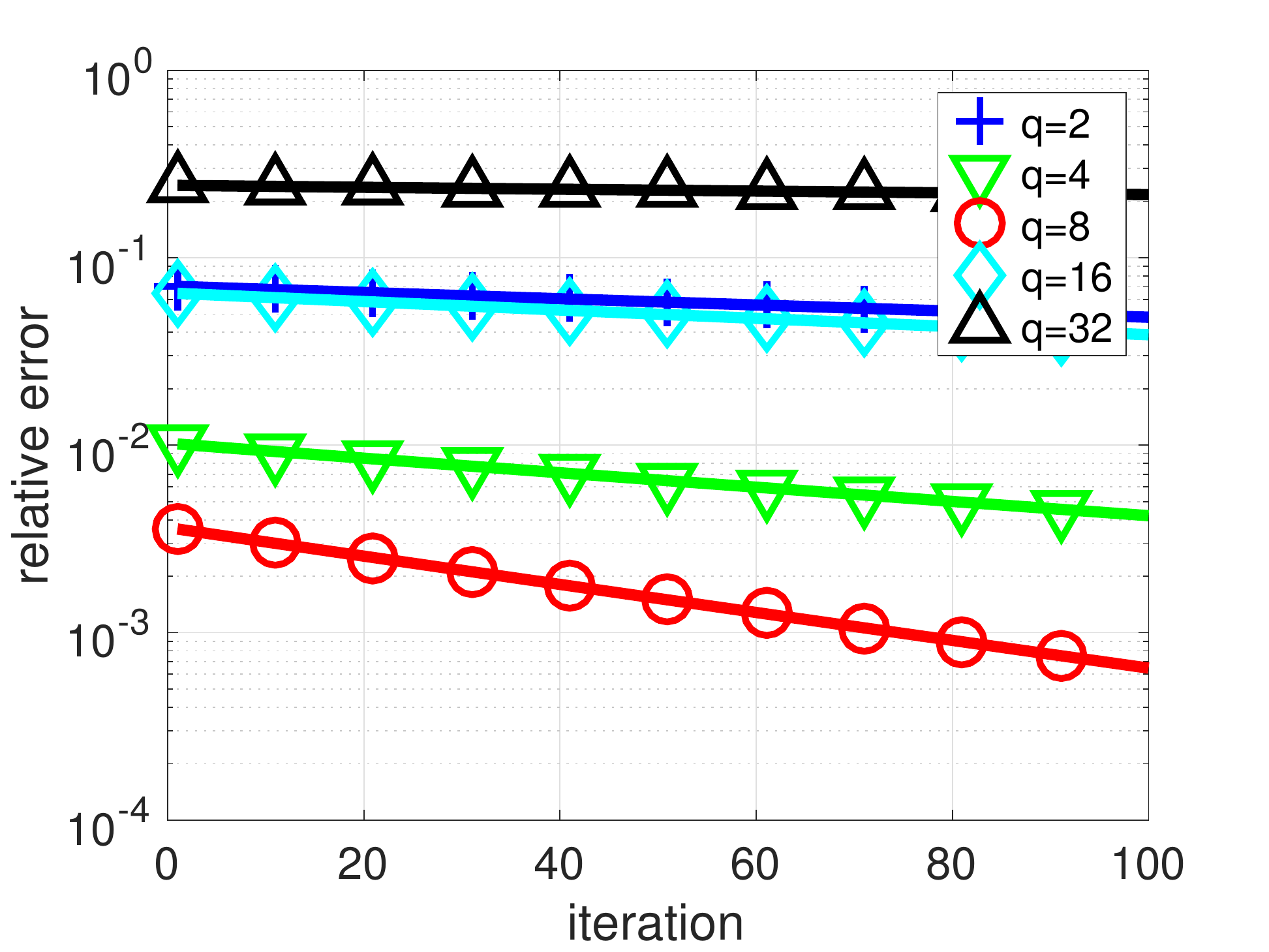}}
  \caption{ RE on the semi-log scale  for the $128\times 128 $ RPP of angle range $[0,2\pi]$ vs 100 AP iterations
  after initialization given by 300 DR iterations with various $q$.
  }  
  \label{fig:different-q}
\end{figure}
 
Fig. \ref{fig:different-q} shows RE versus 100 AP iterations
after the DR initialization with an i.i.d. random mask and two Fresnel masks with various $q$. Interestingly, we observe  that a mask of higher complexity (random or larger $\rho$) works better with a smaller value of $q$. But $q$ is as large as $32$, the results are always poor.

\subsection{Effect of noise}

\begin{figure}
  \subfigure[Random with $50\%$ overlap]{\includegraphics[width=5cm]{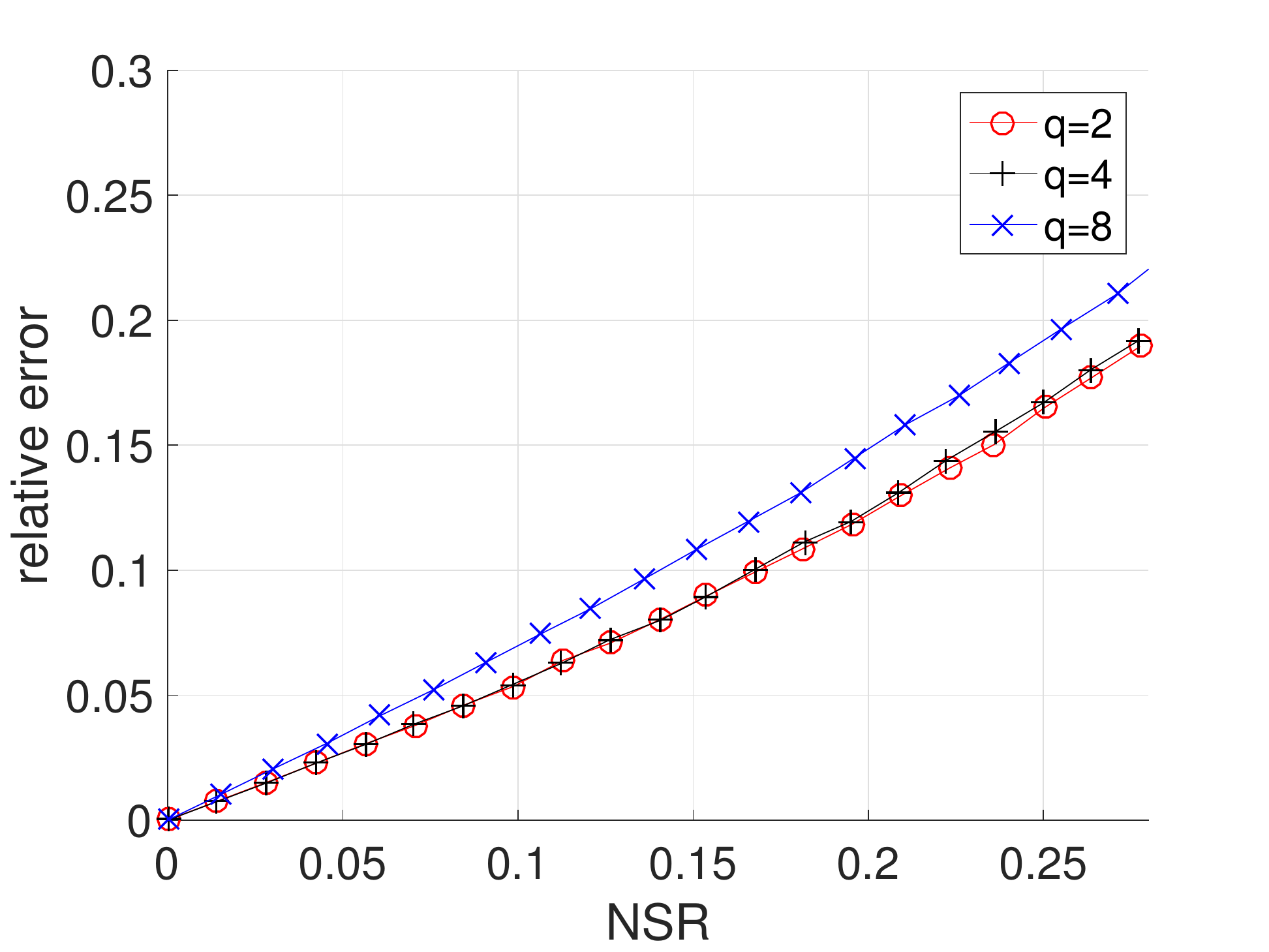}} 
      \subfigure[Fresnel with $\rho= {6\over 5\pi}$]{\includegraphics[width=5cm]{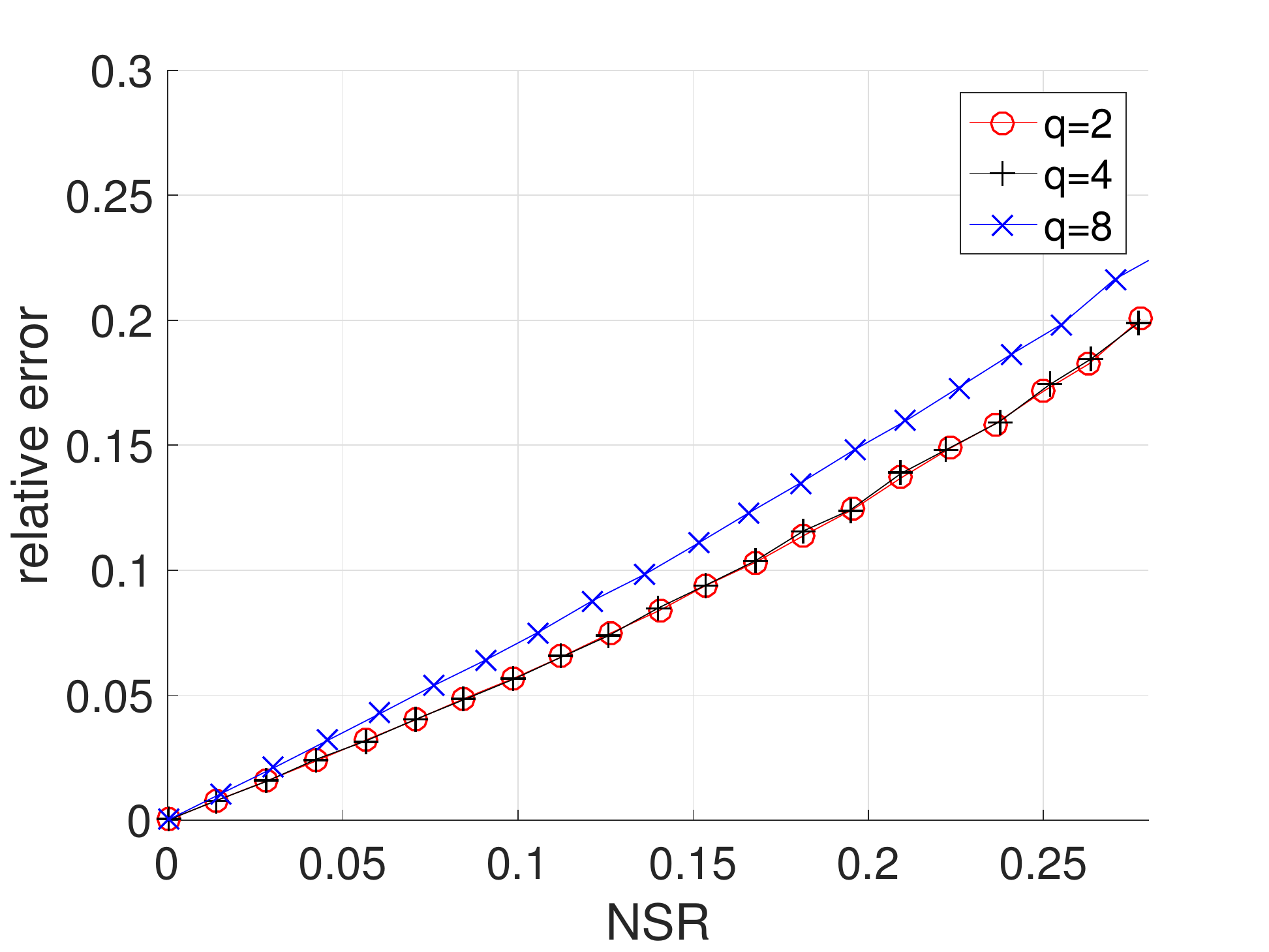}} 
       \subfigure[Fresnel with $\rho={3\over 25\pi}$]{\includegraphics[width=5cm]{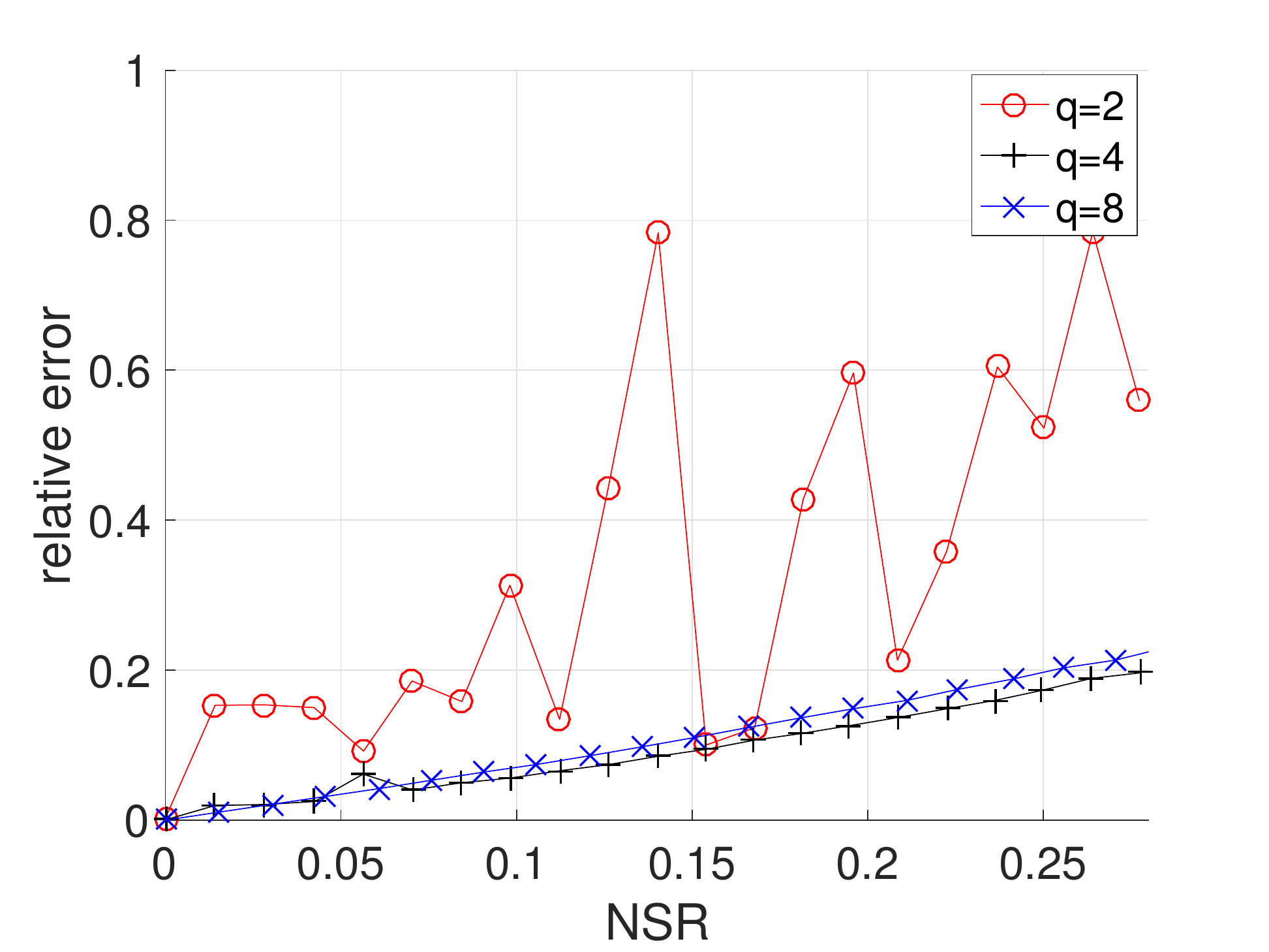} }
 \subfigure[Random with $75\%$ overlap]{\includegraphics[width=5cm]{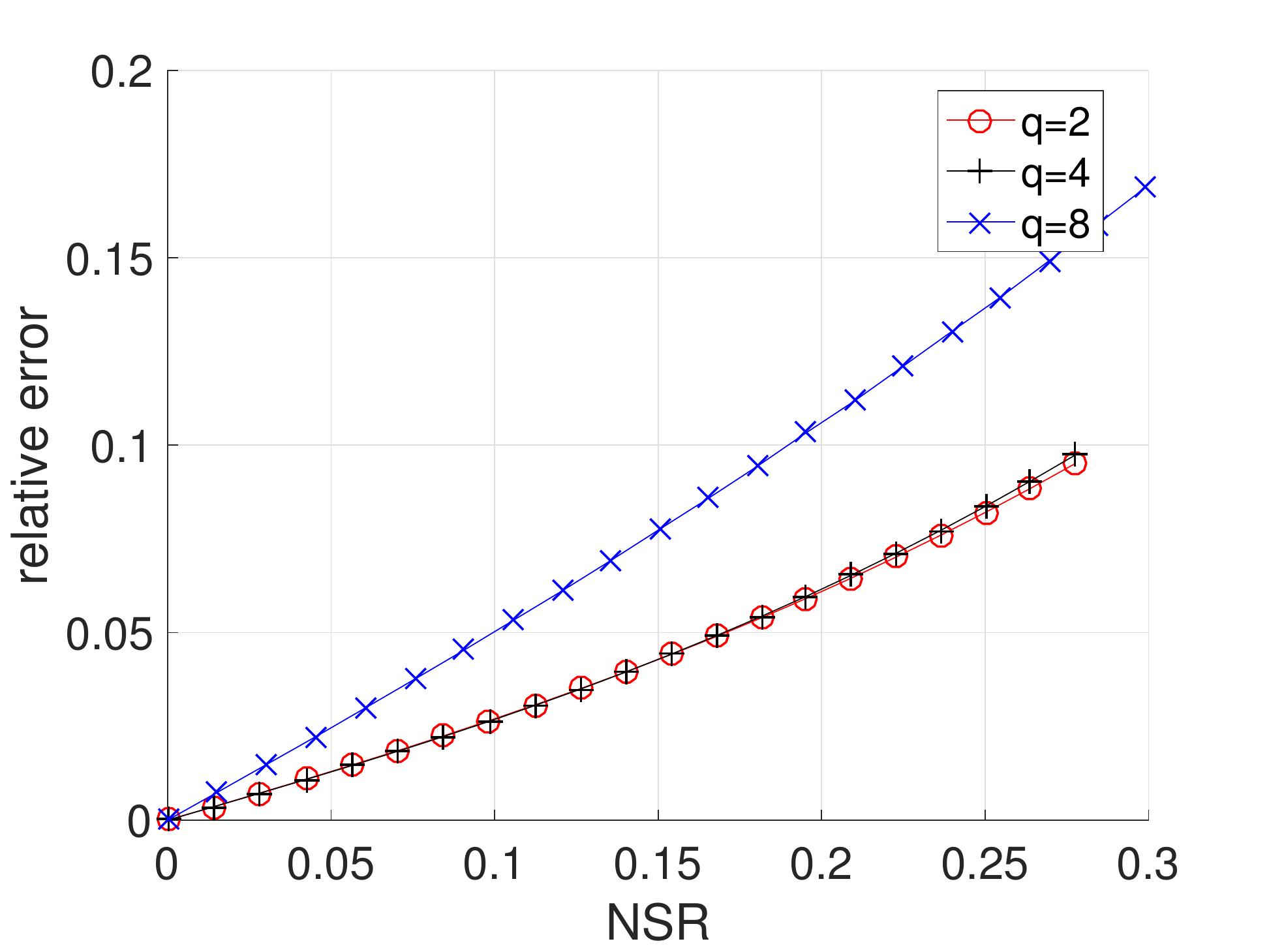}} 
 \subfigure[Fresnel with $\rho= {6\over 5\pi}$]{\includegraphics[width=5cm]{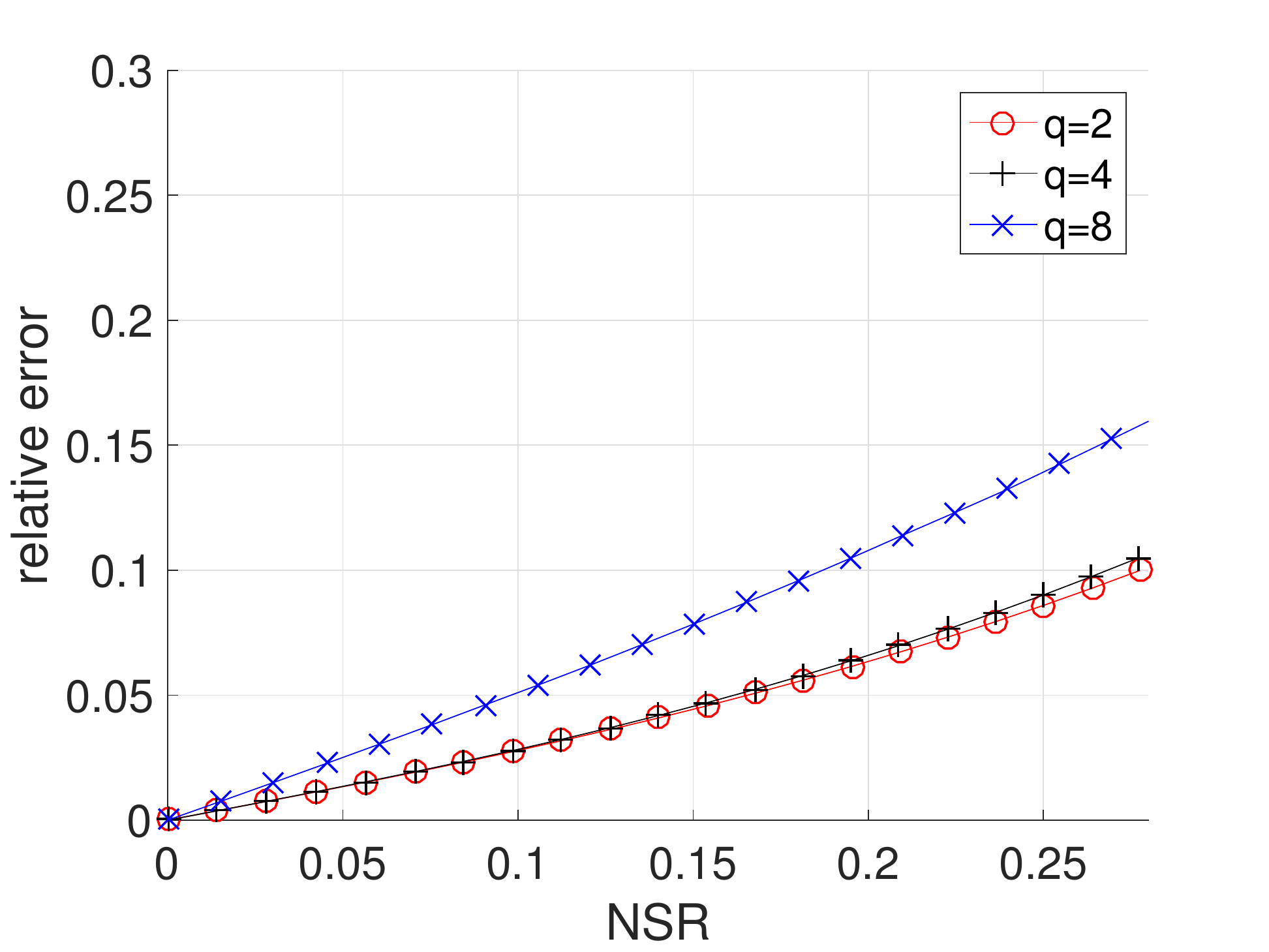}} 
 \subfigure[Fresnel with $\rho={3\over 25\pi}$]{\includegraphics[width=5cm]{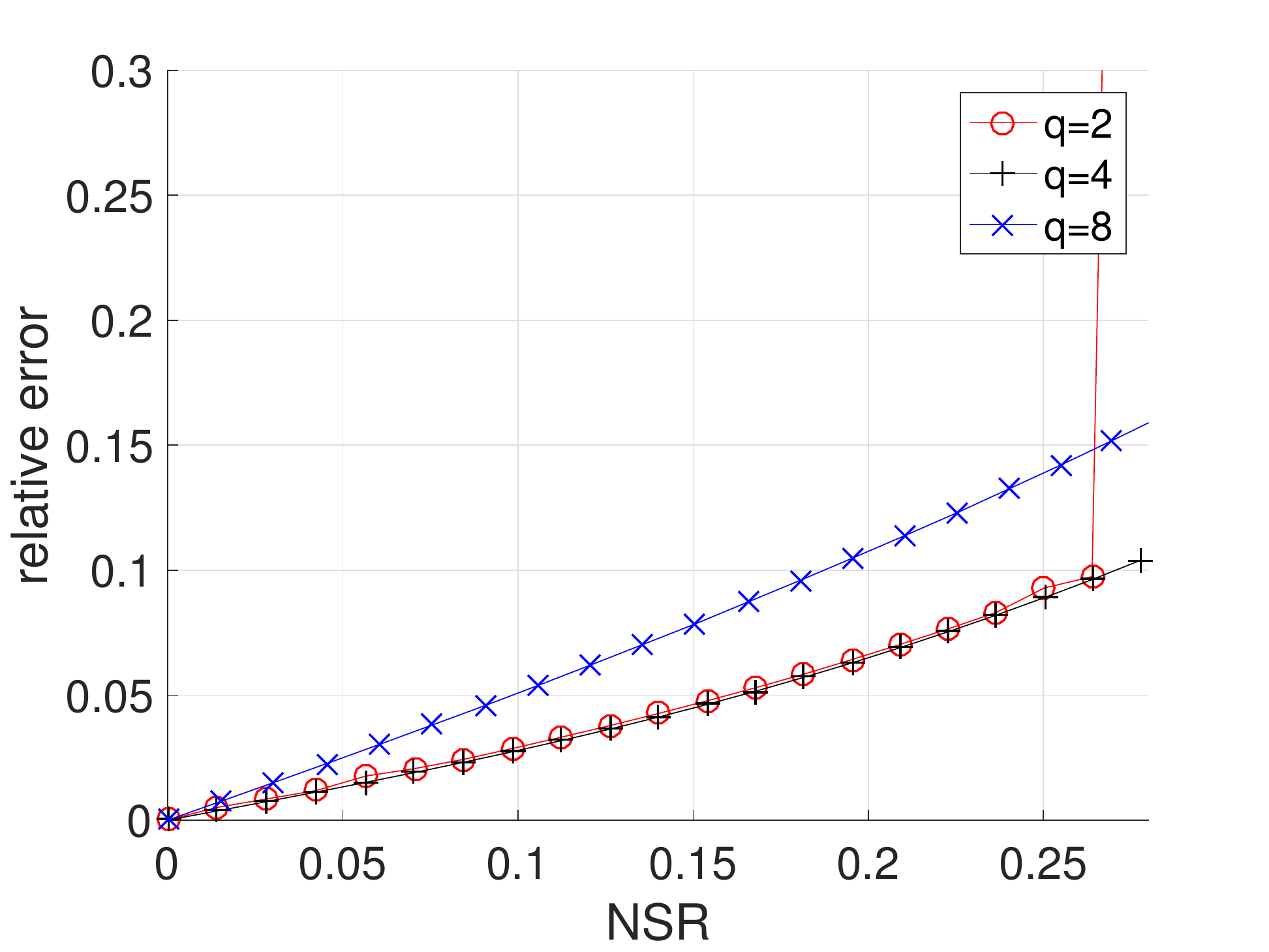}}
  \caption{ RE vs. NSR after 300 FDR iterations followed by 200 AP iterations for
   $64\times 64$ RPP of angle range $[0,2\pi]$ with  (a)-(c) $50\%$ and (d)-(f) $75\%$ overlap between adjacent masks.}  \label{fig:nsr1}\label{fig:nsr2}
    \end{figure}

To introduce both phase and magnitude noises to our signal model, 
we add complex Gaussian noise to $A^* f$ before taking the modulus as data
\[
b=\lt|A^* f+ z\rt|
\]
where $z\in \IC^N$ is an i.i.d.  circularly symmetric complex Gaussian random vector. 
The size of the noise is measured in terms of the noise-to-signal ratio (NSR)
\beq
\mbox{\rm NSR}= \frac{\|b-|A^* f|\|}{\|A^* f\|}
\eeq

We also test the effect of increasing the overlap of adjacent masks
from $50\%$ to $75\%$: 
With the same relation $m=2n/q$, $75\%$ overlap between adjacent masks 
corresponds to $4q^2$ diffraction patterns.

Fig. \ref{fig:nsr1} shows  RE versus NSR for various
masks and $q$ with (top) $50\%$ and (bottom) $75\%$ overlap. 
The correlated  masks  used are shown in Fig. \ref{fig:mask} (d)-(f)  with $m/\ell=4$ so that the complexity of the mask is independent of $q$. 

For up to $25\%$ NSR, the RE-NSR curves in Fig. \ref{fig:nsr1} are roughly straight lines of a slope less than 1, with $q=4$ the best performing value across the board. 
The violent fluctuations in (c) for $q=2$ indicates non-convergent behaviors consistent with Fig. \ref{fig:different-q}(c) with $q=2$.  

Fig. \ref{fig:nsr1}(d)-(f) shows that the four times number of data with $75\%$ overlap predictably result in a significantly reduced RE, especially for $q=2,4$.

\section{Conclusion and discussion}\label{sec:con}
In the present work, we have proved the uniqueness theorem (Theorem \ref{thm:u}, Corollary \ref{cor:u}) for any
ptychographic scheme with an independent random masks under
the minimum overlap condition \eqref{two2}. We have also given a local geometric convergence analysis
for AP and DR algorithms (Theorem \ref{thm1}).
We have shown that  DR has a unique fixed point in the object domain (Proposition \ref{prop4}) and
given a simple criterion for distinguishing the true solution among possibly many fixed points of AP (Proposition \ref{prop:ap}). We chose AP and DR as the building blocks of
our reconstruction algorithm because of the fixed point and convergence properties
and because there are no adjustable parameters which can be tuned to optimize the performance as in other algorithms \cite{ADM, Hesse}. 

\commentout{We choose AP and DR as the building blocks of
our reconstruction algorithm because of the fixed point and convergence analysis given in Section \ref{sec:ITA}
and because there are no adjustable parameters which can be tuned to optimize the performance as in other algorithms \cite{ADM, Hesse}. 
We do not claim that this combination  
yields the best algorithm for ptychography.  Quite the contrary, our approach can be
significantly improved,  for example, in the initialization step where we could have picked
the DR iterate of the least relative residual within a given number of iterations
as the initialization for the subsequent AP iterations. This modification
on initialization, however,  
would put an extra burden on computer memory. 
}

We have proposed a minimalist scheme parametrized by  $q=2n/m$ where $m$ is the number of mask pixels in each direction and given a lower bound on the  geometric rate of convergence (Proposition \ref{prop:gap}). 
The bound $\gamma >1-C/q^2$ predicts a poor performance for  the minimalist scheme  with large $q$ which
is confirmed by our numerical experiments. 

In addition, we have performed extensive numerical experiments to find out  what the general features of a well-performing mask are like, what the best-performing  values of  $q$ for a given mask are,  how robust the minimalist scheme is with respect to measurement noise and what the significant factors affecting
the noise stability are. 

From our numerical experiments, we  have found that 
(i) the mask of higher complexity (e.g. random masks or the Fresnel mask of a larger $\rho$) produces faster convergence in reconstruction than the mask of lower complexity (e.g. the Fresnel mask of a smaller $\rho$);  (ii) the best-performing value of $q$ for a mask of higher complexity is  smaller than that  for a mask of 
lower complexity; (iii) ptychographic reconstruction with medium values of $q$ (e.g. $q\in [4,8]$) is robust 
with respect to measurement noise regardless of the mask used, with the ratio of $\mbox{RE}$ to $\mbox{\rm NSR}$ less than unity;
(iv) increased overlap between adjacent masks generally reduces the reconstruction error as expected.

We have not addressed an important potential of ptychography  for retrieving the mask
and the object simultaneously without knowing precisely the mask function (blind ptychography). We have previously proved the simultaneous determination of a roughly known mask and the object for a nonptychographic setting \cite{pum}.  As a ptychographic reconstruction offers extra potential beyond a nonptychographic one, 
we will turn to blind ptychography in a forthcoming paper.

\appendix

\section{ Twin image with a  Fresnel mask}
We give a sufficient condition for  the existence of twin image with the Fresnel masks
that  satisfies the symmetry of conjugate inversion, which we believe explain the
poor performance of Fig. \ref{fig:different-q} (right column) for $c=0$ and
spikes in Fig. \ref{fig1} both for $q=2$. 

Let  $Q_mx$ be the conjugate inversion of $x\in \IC^{m\times m}$, i.e. 
$(Q_m x)_{ij}=\overline{x}_{m+1-i,m+1-j}.$ 
For an even integer $m$, write
\[
x=\lt[\begin{matrix}
x_1&x_2\\
x_3&x_4
\end{matrix}
\rt],\quad x_j\in \IC^{m/2\times m/2},\quad j=1,2,3,4,
\]
and we have
\[
Q_m x=\lt[\begin{matrix}
Q_{m/2}x_4&Q_{m/2} x_3\\
Q_{m/2}x_2&Q_{m/2}x_1
\end{matrix}
\rt].
\]
For ease of notation, we will omit writing the subscript in $Q$. 
 \commentout{
 \[
 Q_mx=
Q \left[
\begin{array}{cccc}
 x_{1,1} & x_{1,2} & \ldots & x_{1,m} \\
x_{2,1} & x_{1,2} & \ldots & x_{2,m} \\
\vdots  & \vdots & \vdots & \vdots \\
x_{m,1} & x_{m,2} & \ldots & x_{m,m}
\end{array}
\right]
=
 \left[
\begin{array}{cccc}
 \overline{x}_{m,m} & \overline{x}_{m,m-1}& \ldots & \overline{x}_{m,1} \\
 \overline{x}_{m-1,m} & \overline{x}_{m-1,m-1} & \ldots & \overline{x}_{m-1,1} \\
\vdots  & \vdots & \vdots & \vdots \\
 \overline{x}_{1,m} & \overline{x}_{1,m-1} & \ldots & \overline{x}_{1,1}
\end{array}
\right].
 \] %
 }
\begin{prop}\label{fo} 
\commentout{Let $m/2=n/q\in \IZ$ be a positive integer. Let $Q$ be the flipped-over operator for $\IC^{m/2\times m/2}$.
Let $\beta_1,\beta_2$ be  some real numbers.
 Consider the type-2 $m/2\times m/2$ mask 
\[ g(k_1,k_2; \beta_1, \beta_2)=\exp\{2\pi ci [(k_1-\beta_1)^2+(k_2-\beta_2)^2] (m/2)^{-1}\}, \; k_1,k_2=1,\ldots, m.\]
}
Let $\rho\in \IZ$ and  $\mu\in \IC^{m\times m}$ be the Fresnel mask with the elements
\beq
\label{65}
\mu^0(k_1,k_2)=\exp\lt\{\im \pi \rho ((k_1-\beta_1)^2+(k_2-\beta_2)^2)/m\rt\}, \quad k_1, k_2=1,\ldots, m.
\eeq
For an even integer $m$, the matrix  \beq \label{Hform}\overline{Q \mu }\odot \mu:=
\left(
\begin{array}{ccc}
  h_1&  h_2   \\
  h_3 &   h_4   
\end{array}
\right),\quad h_j\in \IC^{m/2\times m/2},\quad j=1,2,3,4,\eeq
satisfies the symmetry
\beq
\label{68} h_1=h_4=\gamma h_2=\gamma h_3,\; \gamma=(-1)^{\rho+2cm}.
\eeq

\end{prop}
\begin{proof}
With
\beq
 \mu=\left(
\begin{array}{cc}
 \mu_1 & \mu_2 \\
\mu_3  & \mu_4
\end{array}
\right),\quad \mu_j\in \IC^{m/2\times m/2},\quad j=1,2,3,4,\eeq
we have 
\[
\overline{Q\mu}\odot \mu=\lt[\begin{matrix}
\overline{Q \mu_4}\odot \mu_1& \overline{Q \mu_3}\odot \mu_2\\
\overline{Q \mu_2}\odot \mu_3&\overline{Q \mu_1}\odot \mu_4
\end{matrix}\rt]
\]
\commentout{
with 
 \begin{eqnarray}
&&\mu_1=[g(k_1,k_2; \beta_1, \beta_2): k_1,k_2=1,\ldots, m/2],\\
&&\mu_2=[g(k_1,k_2; \beta_1+m/2, \beta_2): k_1,k_2=1,\ldots, m/2],\\
&&g_3=[g(k_1,k_2; \beta_1, \beta_2+m/2): k_1,k_2=1,\ldots, m/2],\\
&&g_4=[g(k_1,k_2; \beta_1+m/2, \beta_2+m/2): k_1,k_2=1,\ldots, m/2].
\end{eqnarray}
\begin{eqnarray*}
&&g_4\odot \overline{Q g_1}=\exp\{
2\pi ci [
(k_1-\beta_1-m/2)^2+(k_2-\beta_2-m/2)^2+(m/2-k_1-\beta_1+1)^2+(m/2-k_2-\beta_2+1)^2] (m/2)^{-1}\}
\\
&&
g_1\odot \overline{Q g_4}=\exp\{
2\pi ci [
(k_1-\beta_1)^2+(k_2-\beta_2)^2+(-k_1-\beta_1+1)^2+(-k_2-\beta_2+1)^2] (m/2)^{-1}\}
\\
&&
g_2\odot \overline{Q g_3}=\exp\{
2\pi ci [
(k_1-\beta_1-m/2)^2+(k_2-\beta_2)^2+(m/2-k_1-\beta_1+1)^2+(-k_2-\beta_2+1)^2] (m/2)^{-1}\}
\\
&&
g_3\odot \overline{Q g_2}=\exp\{
2\pi ci [
(k_1-\beta_1)^2+(k_2-\beta_2-m/2)^2+(-k_1-\beta_1+1)^2+(m/2+1-k_2-\beta_2)^2] (m/2)^{-1}\}
\end{eqnarray*}
}
Direct algebra with \eqref{65} and $\rho\in \IZ$ gives 
\begin{eqnarray*}
&&\mu_4\odot \overline{Q \mu_1}=
\mu_2\odot \overline{Q \mu_3} 
\exp\{ (\rho+2cm)\pi \im\}=\mu_3\cdot \overline{Q \mu_2}\exp\{(\rho+2mc)\pi\im\}=\mu_1\cdot \overline{Q \mu_4}
\end{eqnarray*}
and hence the desired result.
\end{proof}


Let $\Phi$ be the oversampled Fourier matrix as before. 
Note that the oversampled Fourier magnitude has the symmetry of conjugate inversion:
\beq\label{fact}|\Phi x|=|\Phi Q x|,\quad\forall x\in \IC^{m\times m}.\eeq 
\begin{prop}
\label{prop:twin} Let $q=2$ (hence $m=2n/q=n$). Suppose $\mu$ satisfies the symmetry \eqref{Hform}-\eqref{68}.
Let $h$ be the matrix given in \eqref{Hform}.  
For any  $x\in\IC^{m\times m}$, let
\[
y=Qx \odot \overline{h}.\]
Then $y$ and $x$ produce  the same ptychographic  data for $q=2$ 
with the mask $\mu$. 
\end{prop}
\begin{proof} 
We can write 
\[ y
=\alpha \left[
\begin{array}{cc}
 Qx_4 \odot \overline{h_1} & Qx_3\odot  \overline{ h_2}  \\
 Qx_2 \odot \overline{h_3} & Qx_1 \odot \overline{ h_4}
\end{array}
\right].
\]
Let $R_1$ and $R_2$ be the reflectors defined by
\[
R_1x=\left[
\begin{array}{cc}
 x_2 & x_1  \\
x_4  & x_3
\end{array}
\right],\quad 
R_2x=\left[
\begin{array}{cc}
 x_3 & x_4  \\
x_1  & x_2
\end{array}
\right],
\]
for all $x\in \IC^{m\times m}$. 
Let 
\beqn
y'&= &(QR_1 x)\odot \overline{R_1h},\\
y''&=&(QR_2 x)\odot \overline{R_2 h},\\
y'''&=&(QR_2R_1 x)\odot \overline{R_2R_1h}. 
\eeqn
We have from direct calculation that 
   \beqn
  |\Phi(x\odot \mu)|&=&  |\Phi(y\odot \mu)|,\\
  |\Phi(R_1x\odot \mu)|&=&   |\Phi(y'\odot \mu)|,\\
  |\Phi (R_2x\odot \mu)|&=&  |\Phi(y''\odot \mu)|,\\
  |\Phi(R_2 R_1 x\odot \mu)|&=&  |\Phi(y'''\odot \mu)|.
  \eeqn 
  \commentout{
where
\[
R_1x=\left(
\begin{array}{cc}
 x_2 & x_1  \\
x_4  & x_3
\end{array}
\right),
R_1h=\left(
\begin{array}{cc}
 h_2 & h_1  \\
h_4  & h_3
\end{array}
\right),
y'= QR_1 x\odot \overline{R_1h},
\]
\[
R_2R_1x=\left(
\begin{array}{cc}
 x_4 & x_3  \\
x_2  & x_1
\end{array}
\right),
R_2R_1h=\left(
\begin{array}{cc}
 h_4 & h_3  \\
h_2  & h_1
\end{array}
\right),
y''=QR_2R_1 x\odot \overline{R_2R_1h}
\]
and
\[
R_2x=\left(
\begin{array}{cc}
 x_3 & x_4  \\
x_1  & x_2
\end{array}
\right),
R_2h=\left(
\begin{array}{cc}
 h_3 & h_4  \\
h_1  & h_2
\end{array}
\right),
y'''=QR_2 x\odot \overline{R_2 h}.
\]
}
Proposition~\ref{fo} implies
\beq\label{part1}
h=\gamma R_1 h=\gamma R_2 h=R_2R_1 h=\mu\odot \overline{Q\mu }.
\eeq
Using (\ref{fact}) and (\ref{part1}) we have 
\beq\label{73}
 |\Phi(x\odot \mu)|=|\Phi(Qx\odot Q\mu )|
=|\Phi(Q(x)\odot \overline{h}\odot \mu)|
=|\Phi(y\odot \mu)|.\eeq
Similarly,
 we have
 \begin{eqnarray}
\label{74}&& |\Phi(R_1 x\odot \mu)|=|\Phi(QR_1 x\odot Q\mu)|=|\Phi(QR_1 x\odot\overline{ \gamma R_1 h}\odot \mu)|=
|\Phi(y'\odot \mu)|,\\
\label{75}&& |\Phi(R_2 x\odot \mu)|=|\Phi(QR_2 x\odot Q\mu )|=|F(QR_2 x\odot \overline{\gamma R_2 h}\odot\mu)|=
|\Phi(y''\odot \mu )|,\\
\label{76}&& |\Phi(R_2R_1x\odot \mu)|=|\Phi(QR_2R_1\odot Q\mu )|=|\Phi(QR_2R_1 x\odot \overline{R_2R_1h}\odot \mu)|=
|\Phi(y'''\odot \mu)|.\end{eqnarray}
The lefthand and righthand sides of eq. \eqref{73}-\eqref{76} are precisely the
ptychographic data for $q=2$ with the mask $\mu$. 
\end{proof}

{\bf Acknowledgements.} The research of A. Fannjiang is supported in part by  the US National Science Foundation  grant DMS-1413373. 

\end{document}